\documentclass[11pt]{article}

\usepackage{fullpage}
\usepackage[utf8]{inputenc}
\usepackage{amsmath,amsfonts,amssymb,amsthm,bbm}
\usepackage{stackrel}
\usepackage{enumerate}
\usepackage{graphicx}
\usepackage{subfigure}
\usepackage{mathrsfs}

\RequirePackage[textwidth=20mm]{todonotes}

\newtheorem{theorem}{Theorem}

\newtheorem{proposition}[theorem]{Proposition}

\newtheorem{problem}[theorem]{Problem}
\newtheorem{lemma}[theorem]{Lemma}

\newtheorem{remark}[theorem]{Remark}

\usepackage[utf8]{inputenc}

\usepackage{url}
\usepackage{textcomp}
\usepackage{enumerate}

\usepackage{tikz}
\usepackage{tikz-3dplot}
\usetikzlibrary{hobby, calc, shapes}
\usetikzlibrary{decorations.pathreplacing}
\usetikzlibrary{arrows, arrows.meta}

\usepackage{calc}

\numberwithin{theorem}{section}
\numberwithin{equation}{section}
\numberwithin{figure}{section}

\DeclareMathOperator{\Bin}{Bin}

\newcommand{\ve}{\varepsilon}

\newcommand{\wordone}{\mathbbm{1}}

\newcommand{\calF}{\mathcal{F}}

\newcommand{\PP}{\mathbb{P}}

\newcommand{\NN}{\mathbb{N}}
\newcommand{\ZZ}{\mathbb{Z}}

\newcommand{\TT}{\mathbb{T}}
\renewcommand{\vec}{\overrightarrow}

\usepackage{mathtools}
\newcommand{\ra}[1]{\stackrel{#1}{\rightarrow}}
\newcommand{\lra}{\xrightarrow}
\newcommand{\zgg}[1]{\stackrel{#1}{\leadsto}}
\newcommand{\nzgg}[1]{\stackrel{#1}{\not \leadsto}}

\def\arraypar#1{\parbox[c]{\textwidth - 2cm}{\centering #1}}

\RequirePackage{environ}
\NewEnviron{display}{
  \begin{equation}
    \begin{array}{c}
      \arraypar{\BODY}
    \end{array}
  \end{equation}
}

\RequirePackage{array}
\newcolumntype{e}{>{\displaystyle}r @{\,} >{\displaystyle}c @{\,} >{\displaystyle}l}

\def\clap#1{\hbox to 0pt{\hss#1\hss}}

\def\truncdiv#1#2{((#1-(#2-1)/2)/#2)}
\def\moduloop#1#2{(#1-\truncdiv{#1}{#2}*#2)}
\def\modulo#1#2{\number\numexpr\moduloop{#1}{#2}\relax}

\newcounter{constant}
\newcommand{\nc}[1]{\refstepcounter{constant}\label{#1}}
\newcommand{\uc}[1]{c_{\ref{#1}}}
\usepackage[colorlinks]{hyperref}
\usepackage[figure,table]{hypcap} 
\hypersetup{final}
\hypersetup{
    colorlinks,
    linkcolor={blue!50!black},
    citecolor={blue!50!black},
    urlcolor={blue!80!black}
}

\begin{document}

\title{No exceptional words for Bernoulli percolation}

\author{Pierre Nolin\footnote{City University of Hong Kong; E-mail: \texttt{bpmnolin@cityu.edu.hk}. Partially supported by a GRF grant from the Research Grants Council of the Hong Kong SAR (project CityU11304718).}, Vincent Tassion\footnote{ETH Z\"urich; E-mail: \texttt{vincent.tassion@math.ethz.ch}.}, Augusto Teixeira\footnote{IMPA; E-mail: \texttt{augusto@impa.br}.}}

\date{}

\maketitle

\begin{abstract}
  Benjamini and Kesten \cite{Benjamini_Kesten_1995} introduced in 1995 the problem of embedding infinite binary sequences into a Bernoulli percolation configuration, known as \emph{percolation of words}.
  We give a positive answer to their Open Problem 2: almost surely, all words are seen for site percolation on $\ZZ^3$ with parameter $p = \frac{1}{2}$.

  We also extend this result in various directions, proving the same result on $\ZZ^d$, $d \geq 3$, for any value $p \in (p_c^{\textrm{site}}(\ZZ^d), 1 - p_c^{\textrm{site}}(\ZZ^d))$, and for restrictions to slabs. Finally, we provide an explicit estimate on the probability to find all words starting from a finite box.
  \bigskip

\end{abstract}

\section{Introduction}

\subsection*{Percolation of words}

We consider Bernoulli site percolation on the hypercubic lattice $\ZZ^d$, $d \geq 3$, with parameter $p \in (0,1)$. It is obtained by ``coloring'' at random the vertices of the lattice (with colors $0$ or $1$): each vertex has state $1$ with probability $p$, and $0$ with probability $1-p$, independently of the other vertices. This process displays a phase transition for the existence of an infinite connected component of $1$'s, at a certain critical value $p_c^{\textrm{site}}(\ZZ^d) \in (0, 1)$ of the parameter $p$. It is known from \cite{Campanino_Russo_1985} that the percolation threshold satisfies $p_c^{\textrm{site}}(\ZZ^d) < \frac{1}{2}$, so that at $p = \frac{1}{2}$, infinite connected components of both colors coexist almost surely.

In the present paper, we study the problem known as \emph{percolation of words}, introduced by Benjamini and Kesten \cite{Benjamini_Kesten_1995} in 1995. It pertains to embedding infinite binary words, i.e.\@ sequences of the form $\xi = (\xi_0, \xi_1, \ldots)$ where each $\xi_i \in \{0,1\}$, into the percolation configuration (see Figure \ref{fig:example_word} for an illustration).

\begin{figure}[htbp]
  \centering
  \includegraphics[width=4cm]{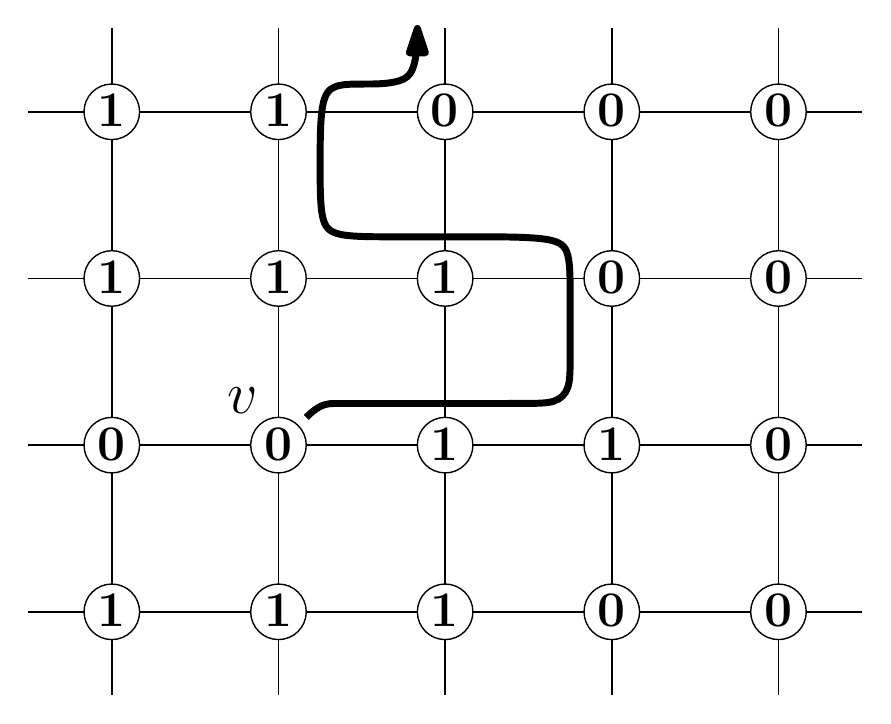}
  \caption{\label{fig:example_word} An embedding of the word $\xi = (0, 1, 1, 0, 1, 1, 1, 0, \ldots)$ into Bernoulli site percolation on the square lattice $\ZZ^2$, starting from a given vertex $v$: we say that the word $\xi$ is seen from $v$.}
\end{figure}

In particular, embedding the constant words $(0, 0, \ldots)$ and $(1, 1, \ldots)$ is equivalent to finding infinite connected components of $0$'s and $1$'s, respectively. Another special case, called \emph{AB percolation}, was introduced earlier by Mai and Halley \cite{Mai_Halley_1980}, motivated by phenomena of polymerization and gelation. In this model, the lattice is randomly populated by ``particles'' of two types, and neighboring particles of opposite types are bonded together (see \cite{Wierman_1989} for a survey). Studying the existence of infinite AB clusters amounts to asking whether the alternating word $(1, 0, 1, 0, \ldots)$ can be embedded. Somewhat surprisingly, it is possible on the triangular lattice at $p = \frac{1}{2}$, as shown by Wierman and Appel \cite{Appel_Wierman_1987}, even though the constant words cannot be embedded (actually, they prove it for all $p$ in the interval $(1 - p_c^{\textrm{site}}(\ZZ^2), p_c^{\textrm{site}}(\ZZ^2))$).

What makes percolation of words particularly challenging to study is its lack of monotonicity. Contrary to usual Bernoulli percolation (i.e. the case of constant words), where one considers events such as ``there exists an infinite connected component of $1$'s'', the events studied in percolation of words are neither increasing nor decreasing in general. For example one can think of the event, from AB percolation, of finding the alternating word $(1, 0, 1, 0, \ldots)$. In particular, the probability that a given non-constant word can be embedded has no clear monotonicity in $p$, and it could even be the case that the set of $p$ for which this probability is positive is not an interval.

\subsection*{Main result: no exceptional words}

Before stating our main result, we need to give a few definitions specific to the study of percolation of words (the more  standard definitions of percolation are presented in Section~\ref{sec:preliminaries}). We denote by \[\Xi := \{0, 1\}^{\mathbb N}\] (where $\mathbb N = \{0, 1, \ldots \}$) the set of all infinite words. For a vertex $v \in \mathbb Z^d$ and a site percolation configuration $\omega$ on $\mathbb Z^d$, we say that a word $\xi \in \Xi$ is \emph{seen from $v$} in $\omega$ if there exists an infinite self-avoiding path $v = v_0 \sim v_1 \sim \ldots$ such that
\begin{equation} \label{eq:see_word}
\omega_{v_i} = \xi_i \:\: \text{for all} \:\: i \geq 0.
\end{equation}
We also introduce the corresponding random set of words $S(v) := \{ \xi \in \Xi \: : \: \xi$ is seen from $v\}$, as well as, for $V \subseteq \ZZ^d$,
$$S_V := \bigcup_{v \in V} S(v) = \big\{ \xi \in \Xi \: : \: \xi \text{ is seen from at least one vertex } v \in V \big\}.$$
Following the notation of \cite{Benjamini_Kesten_1995}, we write $S_{\infty} := S_{\ZZ^d}$: it is the set of words that can be seen somewhere on the lattice.

Benjamini and Kesten proved \cite{Benjamini_Kesten_1995} that for every $\xi \in \Xi$, $\PP_{1/2}(\xi \in S_{\infty}) = 1$. This follows from Wierman's coupling, as explained in Section \ref{sec:Wierman_statement}. Combined with Fubini's theorem, it implies the following property which says, roughly speaking, that the ``uniform random word'' can be seen at $p = \frac{1}{2}$. Consider the product measure $\mu_{1/2} := \big(\frac{1}{2} (\delta_0 + \delta_1) \big)^{\otimes \NN}$ on $\Xi$, then we have
\begin{equation} \label{eq:almost_all}
\PP_{1/2} \big( \mu_{1/2}(S_{\infty}) = 1 \big) = 1.
\end{equation}
In other words, $\PP_{1/2}$-a.s., all ``typical'' $\xi \in \Xi$ can be seen.

This  naturally leads to ask whether there could be some exceptional\footnote{The terminology ``exceptional word'' is inspired from the  work \cite{Schramm_Steif_2010} on exceptional times for critical dynamical percolation. Dynamical percolation is a process indexed by a continuous time parameter $t\in \mathbb R$. At criticality on the triangular lattice, it is known that there is no percolation at every fixed time, and the authors study the existence of possible exceptional times when percolation exists. Here, the situation is analogous, every fixed word can be seen a.s.\@ and we investigate the possibility of possible exceptional words that could not be seen.}  words that cannot be seen, i.e. do we have $S_{\infty} \neq \Xi$ or $S_{\infty} = \Xi$ ($\PP_{1/2}$-a.s.)? This question was stated as Open Problem 2 in \cite{Benjamini_Kesten_1995}, and it remained widely open since then. Our main result gives a complete answer to it.

\begin{theorem} \label{thm:main}
For $d = 3$,
\begin{equation} \label{eq:all_words}
\PP_{1/2}(S_{\infty} = \Xi) = 1.
\end{equation}
\end{theorem}
We prove in fact a stronger statement, see Theorem~\ref{thm:main_strong}. In particular, it is possible to see all words in a (sufficiently thick) slab. Moreover, our result holds not only for $p=\frac{1}{2}$, but also for all $d \geq 3 $ and $p \in (p_c^{\textrm{site}}(\ZZ^d), 1 - p_c^{\textrm{site}}(\ZZ^d))$, i.e.\@ in the regime where we know that two infinite connected components, of $0$'s and $1$'s, coexist.
Finally we establish a quantitative result, namely a precise estimate on the probability that all words can be seen starting from a large ball.

Loosely speaking, we establish Theorem~\ref{thm:main_strong} by constructing all the words simultaneously, thanks to a renormalization argument. We want to emphasize that our proof uses properties of $\ZZ^d$ which are not really specific to this lattice, so we expect it to be quite robust.

\subsection*{Previous results}

We conclude this introduction by mentioning previous works on percolation of words, and related questions.

First, it was shown in \cite{Benjamini_Kesten_1995} (Theorem~1) that \eqref{eq:all_words} holds for all $d \geq 10$. A few years later, Kesten, Sidoravicius and Zhang \cite{Kesten_Sidoravicius_Zhang_2001} proved that \eqref{eq:all_words} also holds for the close-packed graph $\ZZ_{\textrm{cp}}^2$ of the square lattice, obtained from $\ZZ^2$ by adding diagonal edges to each face (for this graph, $p_c^{\textrm{site}}(\ZZ_{\textrm{cp}}^2) < \frac{1}{2}$).

In addition, it was proved in \cite{Benjamini_Kesten_1995} (Theorem~1, again) that for all $d \geq 40$, all words can be seen from the neighbors of one single vertex, i.e.
$$\PP_{1/2}\bigg(\exists v \in \ZZ^d \: : \: \bigcup_{v' \sim v} S(v') = \Xi \bigg) =1,$$
and the same was proved in \cite{Kesten_Sidoravicius_Zhang_2001} for $\ZZ_{\textrm{cp}}^2$. Of course, one cannot hope to see all words $\xi \in \Xi$ from one single vertex $v$, since the value of $\xi_0$ then has to coincide with $\omega_v$. Note that because of this issue, a slightly different convention is adopted in \cite{Benjamini_Kesten_1995}: the condition in \eqref{eq:see_word} is required to hold only starting from $i=1$.

For percolation of words on $\ZZ^d$, $d \geq 3$, the following result (weaker than \eqref{eq:all_words}) was established in \cite{Hilario_Lima_Nolin_Sidoravicius_2014} (Theorem~2), using a renormalization argument. Let $\Xi_M$ stand for the set of words such that all runs of consecutive $0$'s or $1$'s have a length at least $M$ (we allow in particular such words to be ultimately constant), then for all $p \in (p_c^{\textrm{site}}(\ZZ^d), 1 - p_c^{\textrm{site}}(\ZZ^d))$,
$$\exists M = M(d, p) \:\:\: \text{s.t.} \:\:\: \PP_p(\Xi_M \subseteq S_{\infty}) = 1.$$

Finally, the case of site percolation on the triangular lattice $\TT$ at $p = p_c^{\textrm{site}}(\TT) = \frac{1}{2}$ was studied in~\cite{Kesten_Sidoravicius_Zhang_1998}. In this case, the monochromatic words $(0, 0, \ldots)$ and $(1, 1, \ldots)$ cannot be seen, since percolation does not occur, so we have in particular $\PP_{1/2}(S_{\infty} = \Xi) = 0$. However, one can show that \eqref{eq:almost_all} still holds in this case, i.e.\@ that \emph{almost all} words can be seen:
$$\PP_{1/2}(\mu_{1/2}(S_{\infty}) = 1) = 1.$$
This result answered Open Problem 1 in \cite{Benjamini_Kesten_1995}. In this case where $p = \frac{1}{2}$ is critical, one cannot use Wierman's coupling directly, since the lower bound that it provides is simply $0$. Much more work is required for this result than in the case of $\ZZ^d$ ($d \geq 3$), where $p = \frac{1}{2}$ is supercritical.

\subsection*{Organization of the paper}

In Section \ref{sec:preliminaries}, we first set notation, and then we discuss preliminary properties. In particular, we present Wierman's coupling, which is instrumental in the proof of Theorem~\ref{thm:main}, in combination with a renormalization argument.
In Section \ref{sec:proof_main}, we state and prove a stronger and quantitative version of Theorem~\ref{thm:main}, see Theorem~\ref{thm:main_strong}.
This proof uses several auxiliary results, that are established in subsequent sections. We then proceed in detail with the renormalization procedure in Section \ref{sec:k_blocks}. Finally, the last two sections are devoted to proving two inputs used for Theorem~\ref{thm:main_strong}, namely a strengthened version of Wierman's coupling (Section \ref{sec:Wierman}), and a ``propagation'' lemma for oriented percolation (Section \ref{sec:oriented}). Finally, we present in Section~\ref{sec:open} some open problems that arise naturally from this work.

\paragraph{Acknowledgements}
We are very grateful to  Vladas Sidoravicius for introducing us to this problem.
This research started during a visit of AT to ETH Zurich under the support of the FIM.
During this period, AT has also been supported by grants ``Projeto Universal'' (406250/2016-2) and ``Produtividade em Pesquisa'' (304437/2018-2) from CNPq and ``Jovem Cientista do Nosso Estado'', (202.716/2018) from FAPERJ. The research of VT is supported by NCCR SwissMAP, funded by the Swiss NSF.

\section{Preliminaries} \label{sec:preliminaries}

\subsection{Notation}\label{sec:notation}

Let $d \geq 3$. We work with the hypercubic lattice $\ZZ^d$, whose vertices are the points with integer coordinates, and two vertices  are connected by an edge if and only if they are at a Euclidean distance one from each other. Write $u\sim v$ if $u$ and $v$ are two neighbors.

Bernoulli site percolation on $\ZZ^d$, with parameter $p \in (0, 1)$, can be represented as follows. A percolation configuration is of the form $(\omega_v)_{v \in \ZZ^d}$, where each $\omega_v$ (the state of vertex $v$) is either $0$ or $1$. We denote by $\Omega := \{0, 1\}^{\ZZ^d}$ the set of configurations, and we equip it with the cylindrical $\sigma$-algebra $\calF$. We consider on $\Omega$ the product measure $\PP_p := (p \delta_0 + (1-p) \delta_1)^{\otimes \ZZ^d}$, under which vertices have independently state $0$ or $1$, with respective probabilities $1 - p$ and $p$.
This random coloring defines a partition of the lattice: the vertices can be grouped into connected components, or \emph{clusters}, of $0$-vertices and of $1$-vertices.

As the parameter $p$ varies, Bernoulli site percolation displays a phase transition, i.e.\@ a major change of macroscopic behavior, at a certain critical value $p_c = p_c^{\textrm{site}}(\ZZ^d) \in (0,1)$. For all $p < p_c$, there is $\PP_p$-almost surely no infinite cluster of $1$-vertices, while for all $p > p_c$, there is $\PP_p$-almost surely such an infinite cluster, which moreover turns out to be unique. As mentioned in the introduction, we know from \cite{Campanino_Russo_1985} that $p_c^{\textrm{site}}(\ZZ^d) < \frac{1}{2}$. We thus consider a value $p \in (p_c^{\textrm{site}}(\ZZ^d), 1 - p_c^{\textrm{site}}(\ZZ^d))$, for which infinite clusters of both colors coexist almost surely, and we can assume without loss of generality that $p \in (p_c^{\textrm{site}}(\ZZ^d), \frac{1}{2}]$.

Recall that we denote by $\Xi := \{0,1\}^{\NN}$ the set of infinite words, and (for $v \in \ZZ^d$ and $V \subseteq \ZZ^d$) the notations $S(v)$, $S_V$, $S_{\infty}$ ($\subseteq \Xi$) from the introduction. For $\ell \geq 1$, we also introduce the set $\Xi_{\ell} := \{0,1\}^{\ell}$ of finite words $\xi = (\xi_0, \ldots, \xi_{\ell-1})$ with length $\ell$. For $\xi \in \Xi$, the finite sub-word between indices $i$ and $j$ ($0 \leq i \leq j$) is denoted by $\xi_{[i, j]} := (\xi_i, \ldots, \xi_j)$ ($\in \Xi_{j-i+1}$).

For two vertices $v, v' \in \ZZ^d$ and an infinite word $\xi \in \Xi$, we denote by $v \zgg{\xi} v'$ the event that there exists $n \geq 1$ and a self-avoiding path $v = v_0 \sim v_1 \sim \ldots \sim v_{n-1} = v'$ from $v$ to $v'$ along which the beginning of $\xi$ is seen, i.e.\@ such that $\omega_{v_i} = \xi_i$ for all $i \in \{0, \ldots, n-1\}$. Note that when $\xi$ is the infinite monochromatic word $\wordone := (1, 1, \ldots)$, $v \zgg{\wordone} v'$ refers to the usual $1$-connectedness between vertices. For $\xi \in \Xi$ and $v \in \ZZ^d$, we often use the notation $v \zgg{\xi} \infty$ for the event that $\xi$ is seen from $v$ (i.e.\@ $\xi \in S(v)$). The percolation probability is denoted, as usual, by $\theta(p) := \PP_p(0 \zgg{\wordone} \infty)$.

We also need the following notion of ``exact'' word-connectedness. For a finite word $\xi \in \Xi_{\ell}$ (for some $\ell \geq 1$), we denote by $v \ra{\xi} v'$ the event that there exists a self-avoiding path $v = v_0 \sim v_1 \sim \ldots \sim v_{\ell-1} = v'$ with length $\ell$ along which $\xi$ is seen.

\subsection{Measurability properties} \label{sec:measurability}

In this section, we recall some useful measurability properties that were established in \cite{Benjamini_Kesten_1995}. The most important one is $\{S_{\infty} = \Xi\} \in \calF$, from Proposition 2 of \cite{Benjamini_Kesten_1995}. A zero-one law thus holds for this event, by ergodicity.

Proving this result requires some work, and it is obtained in \cite{Benjamini_Kesten_1995} from an application of the Baire category theorem. This proof actually yields the following stronger property. If $S_{\infty} = \Xi$, i.e.\@ all words are seen somewhere, then there exists a finite subset of vertices $F \subseteq \ZZ^d$ such that all words are seen from $F$:
\begin{equation} \label{eq:all_words_finite}
\big\{ S_{\infty} = \Xi \big\} = \bigcup_{\substack{F \subseteq \ZZ^d\\ |F| < \infty}} \big\{S_F = \Xi \big\}
\end{equation}
(see Remark 1 in \cite{Benjamini_Kesten_1995}).

For values $p \in (p^{\textrm{site}}_c(\mathbb{Z}^d), 1 - p^{\textrm{site}}_c(\mathbb{Z}^d))$, the constructions used for Theorem~\ref{thm:main} do not provide a ``constructive'' proof of measurability for $\{S_{\infty} = \Xi\}$. However, they do show that there exists a measurable subset of $\{S_{\infty} = \Xi\}$, which has full probability, and on which a finite subset of vertices $F$ as in \eqref{eq:all_words_finite} exists (as well as a quantitative statement, see Theorem~\ref{thm:main_strong} below).

\begin{remark}
Note that the following (easier to establish) property holds as well. For every $\xi \in \Xi$, $\{\xi \in S(v)\} \in \calF$ for all $v \in \ZZ^d$, so $\{\xi \in S_{\infty}\} \in \calF$. Hence, a zero-one law follows readily, again by ergodicity:
\begin{equation} \label{eq:zero_one}
\text{for all } \xi \in \Xi, \quad \PP_p(\xi \in S_{\infty}) \in \{0, 1\}
\end{equation}
(see Proposition 3 of \cite{Benjamini_Kesten_1995}).
\end{remark}

\subsection{Wierman's coupling} \label{sec:Wierman_statement}

We are going to use a construction due to Wierman (see \cite{Wierman_1989}, and also \cite{Lima_2008}) to translate high $1$-connectedness in the supercritical regime into high ``$\xi$-connectedness''. It produces the basic building blocks of our renormalization strategy, taking care of reading words at the microscopic level.
A simple way of stating it is the following: for every $\xi \in \Xi$,
$$\PP_p(0 \zgg{\xi} \infty) \geq \PP_p(0 \zgg{\wordone} \infty) = \theta(p) > 0$$
(recall that we assume $p \in (p_c^{\textrm{site}}(\ZZ^d), \frac{1}{2}]$). Hence, we can use \eqref{eq:zero_one} to get
\begin{equation} \label{eq:every_fixed_word}
\text{for all } \xi \in \Xi, \quad \PP_p(\xi \in S_{\infty}) = 1.
\end{equation}
In a similar way as in the introduction, this implies, using Fubini's theorem, that for every $q \in (0,1)$: $\PP_p$-a.s., $\mu_q$-almost every word $\xi \in \Xi$ can be seen, where $\mu_q$ is the product measure $\mu_q := \big( (1-q) \delta_0 + q \delta_1 \big)^{\otimes \NN}$ on $\Xi$.

\begin{remark}
The property \eqref{eq:every_fixed_word} does not imply, in general, that $\PP_p(S_{\infty} = \Xi) = 1$. A counter-example is provided by Theorem~5 of \cite{Benjamini_Kesten_1995}, on a well-chosen tree (see in particular (7.5) and (7.14) in that paper).
\end{remark}

The construction of Wierman's coupling uses an iterative exploration of the $1$-cluster of a given vertex. For any fixed $\xi \in \Xi$, it constructs two coupled configurations $(\omega, \tilde{\omega})$ such that $\omega, \tilde{\omega} \sim \PP_p$, and having the property that for every $y \in \ZZ^d$,
\begin{equation}
  \omega \in \big\{ 0 \zgg{\wordone} y  \big\} \text{ implies that } \tilde{\omega} \in \big\{ 0 \zgg{\xi} y \big\}.
\end{equation}
We will actually use a generalization of this construction, where we replace the single vertex $0$ by a set $S$ equipped with a set of words $(\xi^{(x)})_{x \in S} \in \Xi^S$.

\begin{lemma} \label{lem:Wierman_gen}
Let $\Lambda \subseteq \ZZ^d$, $S \subseteq \Lambda$, and $(\xi^{(x)})_{x \in S} \in \Xi^S$. There exist two coupled random configurations $(\omega, \tilde{\omega})$ such that $\omega, \tilde{\omega} \sim \PP_p$, and having the following property.
$$\text{For all } y \in \Lambda, \quad \omega \in \big\{ \exists x \in S \: : \: x \zgg{\wordone} y \text{ in } \Lambda \big\} \: \text{ implies that } \: \tilde{\omega} \in \big\{ \exists x \in S \: : \: x \zgg{\xi^{(x)}} y \text{ in } \Lambda \big\}.$$
\end{lemma}

We postpone the proof of Lemma~\ref{lem:Wierman_gen} to Section~\ref{sec:Wierman}.

\section{Main result} \label{sec:proof_main}

In this section we state our main contribution, Theorem~\ref{thm:main_strong}, and we prove it assuming an auxiliary result, Proposition \ref{prop:En}, that will be established in Section \ref{sec:k_blocks}.

\subsection{Statement}

\nc{c:surface_1}

\begin{theorem} \label{thm:main_strong}
Let $d \geq 3$, and $p \in (p_c^{\textrm{site}}(\ZZ^d), 1 - p_c^{\textrm{site}}(\ZZ^d))$.
\begin{itemize}
\item[1.] We have
\begin{equation} \label{eq:main_strong1}
\PP_p (S_\infty = \Xi) = 1.
\end{equation}
Moreover, there exists $l$ ($= l(d,p)$) such that $\PP_p$-a.s., all words can be seen in the slab $\mathbb{Z}^2 \times [0, l]^{d - 2}$.

\item[2.] The following quantitative statement holds, where $B_m(0) := [-m, m]^d$ is the ball with radius $m$ (for the norm $\|.\|_{\infty}$) around $0$ in $\ZZ^d$. There exists a constant $\uc{c:surface_1} = \uc{c:surface_1}(d, p) > 0$ such that:
\begin{equation} \label{eq:main_strong2}
\text{for all } m \geq 1, \quad \PP_p \big( S_{B_m(0)} = \Xi \big) \geq 1 - e^{- \uc{c:surface_1} m^{d - 1}}.
\end{equation}
\end{itemize}
\end{theorem}

Note that \eqref{eq:main_strong1} follows immediately from \eqref{eq:main_strong2}, by letting $m \to \infty$. Also, an analogous upper bound clearly holds in \eqref{eq:main_strong2}, since it is known to be the order of magnitude (up to a different choice of $c_0$) for $\PP_p ( B_m(0) \zgg{\smash{\wordone}} \infty )$.

\begin{remark}
There is nothing special about the fact that we are using two colors. For example, let $I \geq 3$ and consider the process (still in dimension $d \geq 3$) where the vertices are colored independently $0, \ldots, I-1$, with respective probabilities $p_0, \ldots, p_{I-1} \in [0,1]$ ($p_0 + \ldots + p_{I-1} = 1$). If all these parameters are chosen to be strictly larger than $p_c^{\textrm{site}}(\ZZ^d)$ (which is possible if $d$ is large enough), then Wierman's coupling can easily be adapted, and a result analogous to Theorem \ref{thm:main_strong} holds. These parameters $(p_i)_{0 \leq i \leq I-1}$ do not even need to be space-homogeneous: if they are uniformly bounded from below by $p_c^{\textrm{site}}(\ZZ^d) + \ve$, for some $\ve > 0$, then the same arguments work. Finally, let us mention that the methods in our proof of Theorem \ref{thm:main_strong} could also be used to construct point-to-point $\xi$-connections.
\end{remark}

\subsection{Proof of Theorem \ref{thm:main_strong}} \label{sec:proof_main_strong}
\newcommand {\Ct}{C}

For some given $\delta > 0$ and $k \geq 1$, which we explain how to choose in Proposition~\ref{prop:En} below, we write $\Ct := \big| [-k,k]^d \big| = (2k+1)^d$.
In the slab
$${\mathbb S}_h := \ZZ^2 \times (0, hk] \times (-k, k]^{d-3} \subseteq \ZZ^d \quad (h \geq 2),$$
we consider the finite box
$$\Lambda_n := [-kn, kn]^2 \times (0, hk] \times (-k, k]^{d-3} \quad(n\ge1),$$
as well as its inner vertex boundary $\partial \Lambda_n$ with respect to the slab, i.e.
$$\partial \Lambda_n := \big\{ x \in \Lambda_n \: : \: x \sim y \text{ for some } y \in {\mathbb S}_h \setminus \Lambda_n \big\}$$
(see Figure~\ref{fig:partial_lambda}).

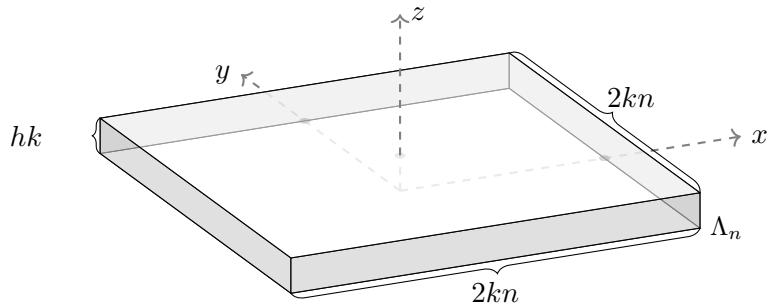
\begin{figure}[h]
  \centering
  \tdplotsetmaincoords{70}{110}
  \begin{tikzpicture}[scale=.5,tdplot_main_coords]
    \tdplotsetrotatedcoords{0}{0}{135}
    \begin{scope}[tdplot_rotated_coords]
      \draw[color=gray,dashed,thick, ->] (0, 0, 0) -- (10, 0, 0) node[right,color=black]{$x$};
      \draw[color=gray,dashed,thick, ->] (0, 0, 0) -- (0, 10, 0) node[left,color=black]{$y$};
      \draw[color=black,very thick] (6, 0, 0) circle (.1);
      \draw[color=black,very thick] (0, 6, 0) circle (.1);
      \draw[color=gray,very thick] (0, 0, 1) circle (.1);
      \draw[color=gray,dashed,thick] (0, 0, 0) -- (0, 0, 1);
      \draw[fill=white, fill opacity=.6] (-6,6,0) -- (6,6,0) -- (6,-6,0) -- (-6,-6,0) -- cycle;
      \draw[fill=gray!40!white, fill opacity=.6] (-6,6,0) -- (6,6,0) -- (6,6,1) -- (-6,6,1) -- cycle;
      \draw[fill=gray!40!white, fill opacity=.6] (-6,-6,0) -- (-6,6,0) -- (-6,6,1) -- (-6,-6,1) -- cycle;
      \draw[fill=gray!40!white, fill opacity=.6] (6,-6,0) -- (6,6,0) -- (6,6,1) -- (6,-6,1) -- cycle;
      \draw[fill=gray!40!white, fill opacity=.6] (-6,-6,0) -- (6,-6,0) -- (6,-6,1) -- (-6,-6,1) -- cycle;
      \draw[fill=white, fill opacity=.6] (-6,6,1) -- (6,6,1) -- (6,-6,1) -- (-6,-6,1) -- cycle;
      \draw[color=gray,dashed,thick, ->] (0, 0, 1) -- (0, 0, 5) node[right,color=black]{$z$};
      \draw[decorate,decoration={brace,mirror,amplitude=5pt}] (-6, -6, 0) -- (6, -6, 0) node [black, midway, yshift=-12pt] {$2kn$};
      \draw[decorate,decoration={brace,mirror,amplitude=5pt}] (6,-6, 1) -- (6,6, 1) node [black, midway, yshift=10pt, xshift=10pt] {$2kn$};
      \draw[decorate,decoration={brace,mirror,amplitude=3pt}] (-6,6, 1) -- (-6,6, 0) node [black, midway, xshift=-28pt] {$hk$};
      \node[right] at (6, -6, 0) {\small $\Lambda_n$};
    \end{scope}
  \end{tikzpicture}
  \caption{\label{fig:partial_lambda} Illustration of the box $\Lambda_n$ for $d = 3$, with $\partial \Lambda_n$ marked in transparent gray.}
\end{figure}

For any given $n > m \geq 1$ and $\xi \in \Xi$, we introduce the event
\begin{equation} \label{eq:5}
  E_m^n(\xi) :=
  \Bigg\{
  \begin{array}{c}
    \exists T \subseteq \partial \Lambda_n, \text{ with } |T| \geq 8 \delta \cdot |\partial \Lambda_n|, \text{ such that: for all } y \in T,\\
    \text{there exists } x \in \partial \Lambda_{m + 1} \text{ with } x \lra{\xi_{[0,t_x]}} y \text{ for some } t_x \leq \Ct n
  \end{array}
  \Bigg\}.
\end{equation}
When $n = m$, we define $E_m^n(\xi)$ to be the whole probability space.

The constants $\delta$ and $k$ are chosen to ensure the existence of microscopic connections as described in Proposition~\ref{prop:En} below, and we think of them as fixed.
The role of $h$ (thickness of the slab) is explained by the computation in the proof of Theorem~\ref{thm:main_strong} below. We denote by $\PP_p^h$ the Bernoulli site percolation measure  on the slab ${\mathbb S}_h$.

\begin{proposition} \label{prop:En}
  For every $p \in (p_c^{\textrm{site}}(\ZZ^d), 1 - p_c^{\textrm{site}}(\ZZ^d))$, there exist $\delta,c > 0$ (small enough) and $k \geq 1$ (large enough) such that the following holds.
  For every $n \geq m \geq 1$, $h \geq 2$  and $\xi \in \Xi$,
\begin{equation}
  \PP_p^h \big( E_m^n(\xi) \setminus E_m^{2n}(\xi) \big) \leq e^{- c h n}.
\end{equation}
\end{proposition}

Proposition~\ref{prop:En} is really the main technical result of this paper, and its proof is postponed to the next section. Assuming its validity, our main result follows easily, as we explain now.

\begin{proof}[Proof of Theorem~\ref{thm:main_strong}]
  Pick $k$ and $\delta$ as in Proposition~\ref{prop:En} and fix some integer $m$.
  For $n = m$, the event $E_m^n(\xi)$ was defined to be the whole space, therefore $\PP_p^h(E_m^n(\xi)) = 1$.

  We now observe that for every word $\xi \in \Xi$, the event $\bigcap_{j \geq 1} E_m^{2^j m}(\xi)$ implies in particular that $\partial \Lambda_{m + 1} \zgg{\smash{\xi}} \partial \Lambda_{2^j m}$ for all $j \geq 1$, from which it follows (as one can easily convince oneself, using a diagonal argument) that $\partial \Lambda_{m + 1} \zgg{\xi} \infty$.
  We can thus write
  \begin{align*}
    \PP_p^h \Big( \exists \xi \in \Xi \text{ such that } \partial \Lambda_{m + 1} \nzgg{\xi} \infty \Big)
    & \leq \PP_p^h \Big( \bigcup_{n \geq m} \bigcup_{\xi \in \Xi_{2\Ct n}} \big( E_m^n(\xi) \setminus E_m^{2n}(\xi) \big) \Big)\\
    & \leq \sum_{n \geq m} \ \sum_{\xi \in \Xi_{2\Ct n}} \PP_p^h \big( E_m^n(\xi) \setminus E_m^{2n}(\xi) \big),
  \end{align*}
  where we used the fact that the event $E_m^n(\xi) \setminus E_m^{2n}(\xi)$ only depends on the first $2 \Ct  n$ coordinates of the word $\xi$ (this is a key observation in order to apply the union bound).
  Now, it follows from Proposition \ref{prop:En} that
  \begin{equation}
    \PP_p^h \Big( \exists \xi \in \Xi \text{ such that } \partial \Lambda_{m + 1} \nzgg{\xi} \infty \Big) \leq \sum_{n \geq m} 2^{2\Ct n} \cdot  e^{- c h  n},
  \end{equation}
  and we can thus choose $h_0$ large enough so that the latter probability is at most $e^{- m}$.
  By taking the $m \to \infty$ limit, we obtain part 1 of the statement, with $l = h_0 k$.

Let us now turn to part 2, the quantitative statement. For that, it is enough to observe that by ``slicing up'' the box $B_m(0)$, we can find at least $c' m^{d - 2}$ disjoint slabs, of height $h_0 k$, whose intersection with this box is a graph isomorphic to $\Lambda_m$, for some constant $c' > 0$. Hence, we have these many independent attempts to see all words starting from $B_m(0)$ in a slab. Since each of these attempts fails with a probability at most $e^{- m}$, from the proof of part 1, we obtain the desired lower bound~\eqref{eq:main_strong2}.
\end{proof}

\section{Proof of Proposition~\ref{prop:En}} \label{sec:k_blocks}

In this section, we prove Proposition~\ref{prop:En}, which was used in the proof of Theorem~\ref{thm:main_strong}. For that, we use two auxiliary results that are established later: Lemma \ref{lem:Wierman_gen} and Proposition \ref{prop:oriented}, proved in Sections~\ref{sec:Wierman} and \ref{sec:oriented} respectively.

\subsection{Renormalization setting} \label{sec:renorm-sett}

\newcommand{\defn}{\textit}

 We will use a dynamic renormalization procedure similar to the standard one presented in \cite[Section 7.2]{Grimmett_book}. Before describing this process, we introduce the necessary geometric framework. For a given word $\xi \in \Xi$, the idea is to construct a ``$\xi$-cluster'' in the slab $\mathbb S_h$ by propagating ``seeds'', i.e.\@ sets which have a reasonably good probability to be $\xi$-connected to infinity, from boxes to boxes (in an oriented way) using local connections. Let $k$ be a fixed even integer.

For $h \geq 1$, we define the (oriented) renormalization graph $\overrightarrow{\mathsf{Slab}}_h = (\mathsf V, \vec{\mathsf E})$, where
\begin{equation} \label{eq:vertices_slab}
\mathsf V := \Big\{ \mathsf{v} = (\mathsf{v}_1, \mathsf{v}_2,\mathsf{v}_3) \in \ZZ^3 \: : \: 0 < \mathsf{v}_3 < h, \: \mathsf{v}_1 \in 2 \ZZ, \: \frac{\mathsf{v}_1}{2} + \mathsf{v}_2 \in 2 \ZZ \text{ and } \frac{\mathsf{v}_1}{2} + \mathsf{v}_3 \in 2 \ZZ \Big\}
\end{equation}
($\subset \ZZ^2 \times (0, h)$), equipped with the oriented edges
\begin{equation} \label{eq:oriented_edges}
\vec{\mathsf u\mathsf v} \in \vec{\mathsf E} \quad \text{if and only if} \quad \mathsf v \in \mathsf u + \big\{ (2, \ve, \ve') \: : \: \ve, \ve' \in \{ \pm 1\} \big\}
\end{equation}
(see Figure~\ref{fig:renormalization}). We explore an oriented site percolation on this graph, that will later be coupled with our microscopic process.

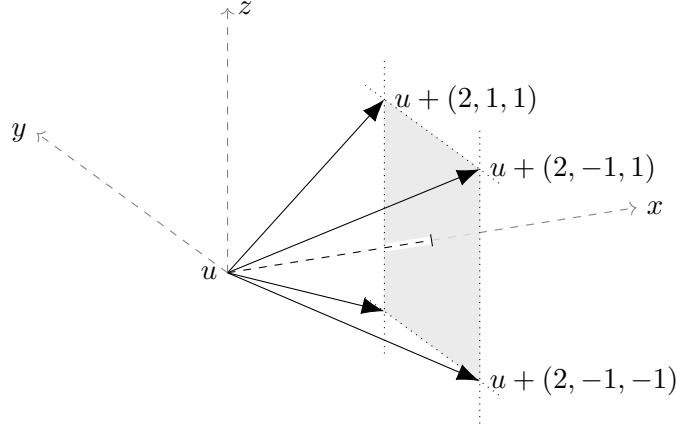
\begin{figure}[h]
  \centering
  \tdplotsetmaincoords{70}{110}
  \begin{tikzpicture}[scale=1.5,tdplot_main_coords]
    \tdplotsetrotatedcoords{0}{0}{135}
    \begin{scope}[tdplot_rotated_coords]
      \draw[color=gray, ->,dashed] (2, 0, 0) -- (4, 0, 0) node[right,color=black]{$x$};
      \draw[fill,color=gray!20,opacity=.8] (2, 1, 1) -- (2, -1, 1) -- (2, -1, -1) -- (2, 1, -1) -- cycle;
      \draw[line width=3pt,color=white] (0, 0, 0) -- (2, 0, 0);
      \draw[color=black,-|,dashed] (0, 0, 0) -- (2, 0, 0);
      \draw[color=gray, ->,dashed] (0, 0, 0) -- (0, 4, 0) node[left,color=black]{$y$};
      \draw[-{Latex[length=3mm]}] (0, 0, 0) -- (2, 1, 1) node[right] {$u + (2, 1, 1)$};
      \draw[-{Latex[length=3mm]}] (0, 0, 0) -- (2, 1, -1);
      \draw[-{Latex[length=3mm]}] (0, 0, 0) -- (2, -1, 1) node[right] {$u + (2, -1, 1)$};
      \draw[-{Latex[length=3mm]}] (0, 0, 0) -- (2, -1, -1) node[right] {$u + (2, -1, -1)$};
      \draw[color=gray, ->,dashed] (0, 0, 0) -- (0, 0, 2.5) node[right,color=black]{$z$};
      \draw[dotted] (2,  1, -1.4) -- (2,  1, 1.4);
      \draw[dotted] (2, -1, -1.4) -- (2, -1, 1.4);
      \draw[dotted] (2, -1.4, 1) -- (2,  1.4, 1);
      \draw[dotted] (2, -1.4, -1) -- (2, 1.4, -1);
      \node[left] at (0, 0, 0) {$u$};
    \end{scope}
  \end{tikzpicture}
  \caption{
    We consider the renormalization graph $\protect\overrightarrow{\mathsf{Slab}}_h$, where each vertex $u$ is connected to the (at most four) vertices $u + \{ (2, \pm 1, \pm 1) \}$.
  }
  \label{fig:renormalization}
\end{figure}

For ${\mathsf U} \subseteq {\mathsf V}$, we define
\begin{equation}
  \mathsf{U}^+ := \bigcup_{\mathsf u\in \mathsf U} \Big\{ \mathsf v \in {\mathsf V} \: :\:{\mathsf u} {\mathsf v}\in\vec {\mathsf E} \Big\} \quad \text{and} \quad \partial^+ \mathsf U := \mathsf U^+\setminus \mathsf U.
\end{equation}
We see the graph $\overrightarrow{\mathsf{Slab}}_h$ as a renormalized version of the slab ${\mathbb S}_h$ ($= {\mathbb Z}^2 \times (0, hk] \times (-k, k]^{d-3}$) introduced in Section \ref{sec:proof_main_strong}, scaled by a factor $\frac{1}{2k}$. We call $\overrightarrow{\mathsf{Slab}}_h$ and $\mathbb{S}_h$ the \emph{macroscopic} and \emph{microscopic} slabs, respectively. Note that the macroscopic graph is oriented, while the microscopic graph is not. In order to distinguish between the two graphs, we use the following typographic convention. For vertices in $\mathbb{S}_h$, we write $x, y, z, \ldots$, and subsets of vertices are denoted by $A, B, C, \ldots$ For vertices in $\overrightarrow{\mathsf{Slab}}_h$, we write ${\mathsf u}, {\mathsf v}, {\mathsf w}, \ldots$, and subsets of vertices are denoted by ${\mathsf A}, {\mathsf B}, {\mathsf C}, \ldots$

In order to carry out the renormalization procedure described in the beginning of the section, we now define \emph{boxes}, \emph{seeds}, and \emph{good local connections}.

\begin{description}
\item[Boxes:] With every vertex $\mathsf u\in\mathsf V$ in the macroscopic slab, we associate two subsets of the microscopic slab, a box $B^{\mathsf u} $ and a face $F^{\mathsf u}$, defined as follows. Set $B := (-k,k]^d$, $F := \{-k\}\times (-k,k]^{d - 1}$, and for every $\mathsf u\in\mathsf V$,
  \begin{equation} \label{eq:14}
    B^{\mathsf u} := k.\mathsf u+B \quad \text{and} \quad F^{\mathsf u} := k.\mathsf u+F
  \end{equation}
  (we see $k.\mathsf u $ as an element of $\mathbb Z^d$ by identifying $\mathbb Z^3$ and $\mathbb Z^3\times\{0_{\mathbb Z^{d-3}}\}$).
  By definition, $F^{\mathsf u} $ is disjoint from $B^{\mathsf u}$. Note also that the boxes $B_{\mathsf u}$ are disjoint from each other for different values of $\mathsf{u}$, as illustrated in Figure~\ref{fig:five_boxes}.

\item[Seeds:] For $\delta \in (0, 1)$, a \defn{$\delta$-seed} for ${\mathsf u}=({\mathsf u}_1,{\mathsf u}_2,{\mathsf u}_3)$ is a pair $(S,t)$ where
  \begin{itemize}
  \item $S \subseteq F^{\mathsf u}$, with $|S| \ge \delta |F^{\mathsf{u}}|$,
  \item and $t=(t_x)_{x\in S}$ is a collection of nonnegative integers such that $t_x\le \Ct {\mathsf u}_1$ for every $x \in S$.
    Informally speaking, one can think of $t_x$ as encoding the time at which we arrived at $x$, when reading the word $\xi$.
  \end{itemize}
  We write $\mathscr S_\delta^{\mathsf u}$ for the set of all $\delta$-seeds for $\mathsf u$.

  Intuitively, we will use seeds as follows. Starting from a given subset $L \subseteq \mathbb Z^d$, we want to explore the vertices that can be reached while reading the word $\xi$, which we do by examining the boxes $B^{\mathsf u}$ one after the other.

  Assume that at some given step, the exploration reaches the box $B^{\mathsf u}$.
  At that moment, we can read the word $\xi$ from $L$ to every point in a set $S$ on the face $F^{\mathsf u}$ without using $B^{\mathsf u}$, and we now want to extend the exploration inside this box.
  A \emph{seed} is simply a set $S$ that is sufficiently large to ensure that the exploration has a good probability to continue within that box.

  Notice that knowing the set $S$ alone is not sufficient if one wants to continue the exploration in $B^{\mathsf u}$.
  When we read the word $\xi$ from $L$ to a vertex $x$ in $S$, one needs to keep track of where we stand along $\xi$, i.e.\@ at which index $t_x$, when the exploration reaches $x$.
We bring the reader's attention to the condition $t_x \le \Ct \mathsf{u}_1$ in the definition of a seed, which will play an important role: we want to keep track of the length of the path connecting $L$ to $x$, so that we can control the ``entropy'' created by the words.

\item[Good local connections:] As explained above, we want to propagate an exploration inside a given box $B^{\mathsf u}$.
  We say that this step succeeds when the event below occurs. Let $\delta \in (0, \frac{1}{64000})$. Given a word $\xi \in \Xi$ and a $\delta$-seed $(S^{\mathsf u},t^{\mathsf u})$ for $\mathsf u$, we define the ``good event''
  \begin{equation}
    \label{eq:6}
    G_{\xi}^{\mathsf u}(S^{\mathsf u},t^{\mathsf u}) := \bigcap_{\mathsf v \in \{\mathsf u\}^+ } \Big \{ \exists (S^{\mathsf v},t^{\mathsf v})\in \mathscr S_{64000\delta}^{\mathsf v} \: :\:   \forall y\in S^{\mathsf v}, \: \exists x\in S^{\mathsf u} \text{ with } x \lra{\xi_{[t^{\mathsf u}_x,t^{\mathsf v}_y]}} y \text{ inside $F^{\mathsf u}\cup B^{\mathsf u}$} \Big\}
  \end{equation}
  (see Figure~\ref{fig:five_boxes} for an illustration of the boxes involved).
\end{description}

\begin{figure}[ht]
  \centering
  \tdplotsetmaincoords{70}{110}
  \begin{tikzpicture}[scale=1,tdplot_main_coords]
    \tdplotsetrotatedcoords{0}{0}{135}
    \begin{scope}[tdplot_rotated_coords]
      \draw[->] (1, 0, 0) -- (6, 0, 0) node[right,color=black]{$x$};
      \draw[->] (1, 1, 0) -- (1, 4, 0) node[left,color=black]{$y$};
      \draw[fill=gray!40, fill opacity=.9] (2, -2, -2) -- (2, 0, -2) -- (2, 0, 0) -- (2, -2, 0) -- cycle;
      \draw[fill=gray!40, fill opacity=.9] (2, 0, -2) -- (2, 2, -2) -- (2, 2, 0) -- (2, 0, 0) -- cycle;
      \draw[fill=gray!40, fill opacity=.9] (2, -2, 0) -- (2, 0, 0) -- (2, 0, 2) -- (2, -2, 2) -- cycle;
      \draw[fill=gray!40, fill opacity=.9] (2, 0, 0) -- (2, 2, 0) -- (2, 2, 2) -- (2, 0, 2) -- cycle;
      \draw[fill=gray!20, fill opacity=.9] (2, 2, 2) -- (4, 2, 2) -- (4, 0, 2) -- (2, 0, 2) -- cycle;
      \draw[fill=gray!20, fill opacity=.9] (2, 0, 2) -- (4, 0, 2) -- (4, -2, 2) -- (2, -2, 2) -- cycle;
      \draw[fill=gray!20, fill opacity=.9] (2, -2, 2) -- (4, -2, 2) -- (4, -2, 0) -- (2, -2, 0) -- cycle;
      \draw[fill=gray!20, fill opacity=.9] (2, -2, 0) -- (4, -2, 0) -- (4, -2, -2) -- (2, -2, -2) -- cycle;
      \draw[fill=gray!20, fill opacity=.9] (0, 1, -1) -- (2, 1, -1) -- (2, -1, -1) -- (0, -1, -1) -- cycle;
      \draw[fill=gray!20, fill opacity=.9] (0, 1, 1) -- (2, 1, 1) -- (2, 1, -1) -- (0, 1, -1) -- cycle;
      \draw[fill=black] (1, 1, 0) circle (.03);
      \draw[fill=gray!40, fill opacity=.9] (0, -1, -1) -- (0, 1, -1) -- (0, 1, 1) -- (0, -1, 1) -- cycle;
      \draw[color=black] (1, 0, 0) -- (2, 0, 0);
      \draw[color=black] (1, 0, 0) -- (1, 1, 0);
      \draw[color=black] (1, 0, 0) -- (1, 0, 1);
      \draw[fill=gray!20, fill opacity=.9] (0, 1, 1) -- (2, 1, 1) -- (2, -1, 1) -- (0, -1, 1) -- cycle;
      \draw[fill=gray!20, fill opacity=.9] (0, -1, 1) -- (2, -1, 1) -- (2, -1, -1) -- (0, -1, -1) -- cycle;
      \draw[fill=black] (1, 0, 1) circle (.03);
      \draw[->] (1, 0, 1) -- (1, 0, 2.5) node[left,color=black]{$z$};
      \node[below left,color=black] at (0, -1, -1) {$B^{\mathsf{u}}$};
      \node[below left,color=black] at (2, -2, -2) {$B^{\mathsf{v}}$};
    \end{scope}
  \end{tikzpicture}
  \caption{Example of a box $B^{\mathsf{u}}$ together with its neighbors $\{B^{\mathsf{v}}\}$, $\mathsf{v} \in \{\mathsf{u}\}^+$. The corresponding faces $F^{\mathsf{u}}$ and $(F^{\mathsf{v}})_{\mathsf{v} \in \{\mathsf{u}\}^+}$ are colored in darker gray.}
  \label{fig:five_boxes}
\end{figure}
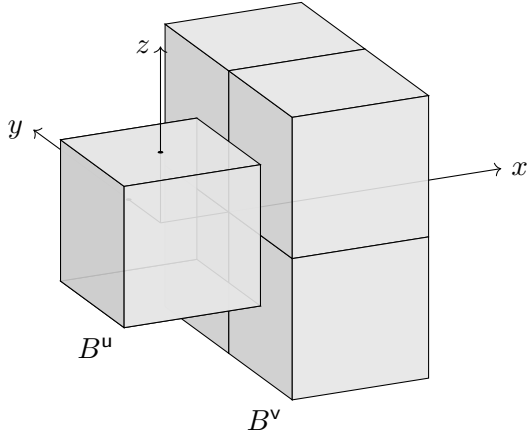

The following proposition ensures that such local connections occur with sufficiently high probability, provided the scaling constant $k$ is chosen large enough.

\begin{proposition} \label{prop:k_blocks} 
Let  $p \in (p_c^{\textrm{site}}(\ZZ^d), 1 - p_c^{\textrm{site}}(\ZZ^d))$ and $\zeta<1$. There exist $0 < \delta(d,p) < \frac{1}{64000}$ small enough, and $k=k(d,p,\zeta) \geq 1$ large enough, such that the following holds. For all $\xi \in \Xi$, $\mathsf u\in \mathsf V$ and $(S,t) \in \mathscr S^{\mathsf u}_{\delta}$,
\begin{equation}
\PP_p \big( G_{\xi}^{\mathsf u}(S,t) \big) \geq \zeta.
\end{equation}
\end{proposition}
We want to emphasize that the choice of  $\delta$ in the proposition above does not depend on $\zeta$.

\begin{proof}[Proof of Proposition \ref{prop:k_blocks}]
  We first observe from Wierman's coupling, Lemma~\ref{lem:Wierman_gen}, that we can restrict ourselves to $\xi = \wordone$, i.e.\@ that it is enough to show
  \begin{equation} \label{eq:proof_k_blocks}
    \PP_p \Big(
    \begin{array}{c}
      \text{$S$ is $1$-connected within $B^{\mathsf{u}}$ to at least $64000\delta |F^{\mathsf{v}}|$ vertices}\\
      \text{in each of the faces $F^{\mathsf{v}}$, $\mathsf{v} \in \{\mathsf{u}\}^+$}
    \end{array}
    \Big) \geq \zeta.
  \end{equation}
  Indeed, this allows one to produce $(S^{\mathsf v}, t^{\mathsf v})\in \mathscr S_{64000\delta}^{\mathsf v}$, since the condition on $t^{\mathsf v}$ is then automatically satisfied, using $| B^{\mathsf{u}} | \leq \Ct$.

  This result follows immediately from classical properties of percolation theory, but we now give a brief sketch of its proof, for the sake of completeness.
  By  symmetry, it suffices to prove \eqref{eq:proof_k_blocks} for a single face $F^{\mathsf{v}}$, instead of all of them.
  Fix now a face $F^{\mathsf{v}}$ with $\mathsf{v} \in \{ \mathsf{u} \}^+$, let $\tilde{F}$ be the intersection of $F^{\mathsf{v}}$ with $B^{\mathsf{u}}$ (which has side length $k$), and let $\hat{F}$ be the middle third of $\tilde{F}$, that is, the vertices of $\tilde{F}$ that are within a distance at least $\frac{k}{3}$ from its boundary.

  Note that for every vertex $x$ in $\hat{F}$, the ball $B(x, \frac{k}{3}) \cap B^{\mathsf{u}}$ is isomorphic to the intersection of $B(0, \frac{k}{3})$ with the half-space $\mathbb{Z}^+ \times \mathbb{Z}^{d - 1}$.
  It is a well-known fact that the critical parameters for site percolation on this half-space and on the full space $\mathbb{Z}^d$ coincide, see Theorem~(7.2), (b) from \cite{Grimmett_book}.


  Let $\theta>0$ be the $\PP_p$-probability that $0$ is connected to infinity in the half-space $\mathbb{Z}^+ \times \mathbb{Z}^{d - 1}$.
By mixing and a second moment estimate, we have
  \begin{equation}
    \PP_p \Big(
    \begin{array}{c}
      \text{at least $\theta |\hat F|/2$ points in $\hat{F}$ are}\\
      \text{$1$-connected to distance $\frac{k}{3}$ within $B^{\mathsf{u}}$}
    \end{array}
    \Big) \xrightarrow[k \to \infty]{} 1.
  \end{equation}
  We now fix $\delta >0$ such that
  \begin{equation}
    \label{eq:12}
    64000 \delta |F| \le \frac{\theta}{2} |\hat F|.
  \end{equation}
  Observe that this choice is independent of $\zeta$ (but it depends on $d$ and $p$). For this choice, the equation above immediately implies
   \begin{equation}
    \PP_p \Big(
    \begin{array}{c}
      \text{at least $64000\delta |F|$ points in $\hat{F}$ are}\\
      \text{$1$-connected to distance $\frac{k}{3}$ within $B^{\mathsf{u}}$}
    \end{array}
    \Big) \xrightarrow[k \to \infty]{} 1.
  \end{equation}
  Similarly, we can prove that 
  
  \begin{equation}
    \inf_{\substack{S \subseteq F^{\mathsf{u}}\\ |S| \geq \delta |F^{\mathsf{u}}|}} \PP_p
    \Big(
    \begin{array}{c}
      \text{some point in $S$ is $1$-connected}\\
      \text{to distance $\sqrt k$ within $B^{\mathsf{u}}\cup F^{\mathsf u}$}
    \end{array}
    \Big) \xrightarrow[k \to \infty]{} 1.
  \end{equation}
  We finally make use of Theorem~3.1 from \cite{Pisztora_1996} to conclude that in a given box with side length $k$, the existence of two distinct $1$-connected components of size at least $\sqrt{k}$, as above, has a vanishing probability as $k \to \infty$. This completes the proof of Proposition \ref{prop:k_blocks}.
\end{proof}

\subsection{Exploration of an oriented percolation in the macroscopic slab}

Let $n \ge 3$ (for simplicity, $n$ is always assumed to be of the form $n=4n'+3$ for some $n'\in \mathbb N$, and $h$ always assumed to be even). Define
\begin{equation} \label{eq:macro_sets}
  \begin{split}
    \mathsf{B}_n & := \mathsf{V} \cap \big( (n, 2n) \times (-2n, 2n) \times (0, h) \big),\\[2mm]
    \mathsf{L}_n & := \mathsf{V} \cap \big( \{n+1\} \times [-n ,n] \times (0, h) \big),\\[2mm]
   \text{and } \mathsf{R}_{n} & := \mathsf{V} \cap \big( \{2n-1\} \times [-n, n] \times (0, h) \big),
  \end{split}
\end{equation}
as depicted on Figure~\ref{f:B_L_R_n}.

\begin{figure}[ht]
  \centering
  \tdplotsetmaincoords{70}{110}
  \begin{tikzpicture}[scale=.5,tdplot_main_coords]
    \tdplotsetrotatedcoords{0}{0}{135}
    \begin{scope}[tdplot_rotated_coords]
      \draw[color=gray!40,dashed,thick, ->] (0, 0, 0) -- (10, 0, 0) node[right,color=black]{$x$};
      \draw[color=gray!40,dashed,thick, ->] (0, 0, 0) -- (0, 10, 0) node[left,color=black]{$y$};
      \draw[color=gray!40,dashed,thick, ->] (0, 0, 0) -- (0, 0, 5) node[right,color=black]{$z$};
      \draw[fill=white, fill opacity=.8] (3,6,0) -- (6,6,0) -- (6,-6,0) -- (3,-6,0) -- cycle;
      \draw[fill=white, fill opacity=.8] (3,6,0) -- (6,6,0) -- (6,6,1) -- (3,6,1) -- cycle;
      \draw[fill=white, fill opacity=.8] (3,-6,0) -- (3,6,0) -- (3,6,1) -- (3,-6,1) -- cycle;
      \draw[fill=white, fill opacity=.8] (6,-6,0) -- (6,6,0) -- (6,6,1) -- (6,-6,1) -- cycle;
      \draw[fill=white, fill opacity=.8] (3,-6,0) -- (6,-6,0) -- (6,-6,1) -- (3,-6,1) -- cycle;
      \draw[rounded corners=.02, fill=gray,fill opacity=.2,draw=black,canvas is yz plane at x=6] (-3, 0) rectangle (3, 1);
      \draw[fill=white, fill opacity=.8] (3,6,1) -- (6,6,1) -- (6,-6,1) -- (3,-6,1) -- cycle;
      \draw[decorate,decoration={brace,mirror,amplitude=5pt}] (0, 0, 0) -- (3, 0, 0) node [black, midway, yshift=-12pt] {$n$};
      \draw[decorate,decoration={brace,mirror,amplitude=5pt}] (3, -6, 0) -- (6, -6, 0) node [black, midway, yshift=-12pt] {$n$};
      \draw[decorate,decoration={brace,mirror,amplitude=5pt}] (6,-6, 1) -- (6,6, 1) node [black, midway, yshift=10pt, xshift=10pt] {$4n$};
      \draw[rounded corners=.02, fill=gray,fill opacity=.2,draw=black,canvas is yz plane at x=3] (-3, 0) rectangle (3, 1);
      \draw[decorate,decoration={brace,mirror,amplitude=4pt}] (6,-6, 0) -- (6,-6, 1) node [black, midway, xshift=10pt] {$h$};
      \node[above] at (3, 0, 0) {\small $\mathsf{L}_n$};
      \node[above] at (6, 0, 0) {\small $\mathsf{R}_n$};
      \node[below] at (3, -6, 0) {\small $\mathsf{B}_n$};
    \end{scope}
  \end{tikzpicture}
  \caption{Illustration of the sets $\mathsf{B}_n$, $\mathsf{L}_n$ and $\mathsf{R}_n$.}
  \label{f:B_L_R_n}
\end{figure}

We also choose an arbitrary ordering of $\mathsf{B}_n$: to fix ideas, we can decide to order the vertices in the lexicographic way, which ensures that $\mathsf{u}=(\mathsf{u}_1,\mathsf{u}_2,\mathsf{u}_3)$ is smaller than $\mathsf{v}=(\mathsf{v}_1,\mathsf{v}_2,\mathsf{v}_3)$ when $\mathsf{u}_1 < \mathsf{v}_1$. It will be simpler to think in terms of this particular ordering since it gives rise to explorations ``layer by layer'', even though the particular choice of an ordering is immaterial for the correctness of the proof.

Fix a set $\mathsf{S} \subseteq \mathsf{L}_n$, and let $\mathsf{S}=\mathsf{U}_0\subseteq \mathsf{U}_1\subseteq \mathsf{U}_2\ldots$ and $\emptyset=\mathsf{V}_0\subseteq \mathsf{V}_1\subseteq \mathsf{V}_2\ldots$ be two growing sequences of subsets of $\mathsf{B}_n$.  Such sequences will appear when we explore  all the (oriented) clusters touching $\mathsf{S}$: the  sequence $(\mathsf U_i)$ corresponds to the exploration of the open vertices in these clusters, while  the  sequence $(\mathsf V_i)$ explores their boundaries.  The sequences arising in this context are defined as follows, where at each step we add at most one vertex to $\mathsf U_i$ (collecting the open vertices connected to $\mathsf S$ ) or $\mathsf V_i$ (the closed boundary).   We say that the sequence $(\mathsf{X}_i)_{i \geq 0} = (\mathsf{U}_i,\mathsf{V}_i)_{i \geq 0}$ is an \emph{exploration sequence from $\mathsf{S}$} (in $\mathsf{B}_n$) if for every $i \ge 0$,
\begin{equation}
  \mathsf{X}_{i+1} =
  \begin{cases}
    \: (\mathsf{U}_i,\mathsf{V}_i) & \text{if $\partial^+\mathsf{U}_i \cap \mathsf{B}_n \subseteq \mathsf{V}_i$},\\[2mm]
    \: (\mathsf{U}_i\cup\{\mathsf{z}_i\},\mathsf{V}_i)\text{ or } (\mathsf{U}_i,\mathsf{V}_i\cup\{\mathsf{z}_i\}) & \text{otherwise},
  \end{cases}
\label{eq:2}
\end{equation}
where $\mathsf{z}_i$ is the minimum vertex that belongs to  $\partial^+\mathsf{U}_i\cap \mathsf{B}_n$, but not to $\mathsf{V}_i$.

Note that the exploration terminates in finite time (since we explore a finite graph), so the growing sequence $(\mathsf{U}_i)_{i \geq 0}$ is ultimately constant: we denote by $\mathsf{U}_\infty$ its end value.

\begin{lemma}\label{lem:criterionPercolation}
  For every $\delta>0$, there exist $\zeta = \zeta(\delta)<1$ and $c = c(\delta)>0$ such that the following holds. Assume that $\mathsf{S} \subseteq \mathsf{L}_n$ with $|{\mathsf{S}}|\ge\delta|\mathsf{L}_n|$, and that $(\mathsf{X}_i)_{i \geq 0} = (\mathsf{U}_i,\mathsf{V}_i)_{i \geq 0}$ is a random exploration sequence from $\mathsf{S}$ in $\mathsf{B}_n$ satisfying: for each $i \ge 0$,
\begin{equation} \label{eq:13}
  \PP \big(\mathsf{V}_{i+1}= \mathsf{V}_i  \:|\:\mathsf{X}_0,\ldots,\mathsf{X}_i\big)\ge \zeta \text{ a.s.}
\end{equation}
Then,
\begin{equation} \label{eq:15}
  \PP \big(|\mathsf{U}_\infty\cap \mathsf{R}_n|\ge \tfrac{1}{1000} |\mathsf{R}_n| \big)\ge 1-e^{-c h n}.
\end{equation}
\end{lemma}

Before proving this lemma, let us state an auxiliary result that is proved in Section~\ref{sec:oriented}. Write $\mathsf P_p$~for the site percolation measure in the slab $\overrightarrow{\mathsf{Slab}}_h$ with parameter $p$: independently,  each vertex is open with probability $p$, closed  with probability $1-p$. For $\mathsf A \subset  \mathsf  V$ and $\mathsf v\in \mathsf V$, we write $\mathsf A \xrightarrow{}{} \mathsf v$ if there exists an open  oriented path of $\overrightarrow{\mathsf{Slab}}_h$ from a vertex of $\mathsf A$ to $\mathsf v$. The notation above corresponds to the existence of paths  in standard oriented site percolation, and it should not be confused with the notation  $x \xrightarrow{\xi} y$ introduced at the end of Section \ref{sec:notation} and corresponding to  percolation of words.
\begin{proposition} \label{prop:oriented}
For every $\delta \in (0,1)$, there exist $\gamma \in (0, 1)$ and $c > 0$ such that the following holds. For all $h \geq 1$ and $n$ large enough, if $\mathsf{S} \subseteq \mathsf{L}_n$ satisfies $|\mathsf{S}| \geq \delta |\mathsf{L}_n|$, then for  oriented site  percolation, we have
$$\mathsf{P}_\gamma \Big( \big| \big\{ \mathsf v \in \mathsf{R}_n \: : \: \mathsf{S} \xrightarrow{}{} \mathsf v \big\} \big| \leq \frac{1}{1000} |\mathsf{R}_n| \Big) \leq e^{- c h n}.$$
\end{proposition}
We postpone the proof of the above proposition to Section~\ref{sec:oriented}.

\begin{proof}[Proof of Lemma~\ref{lem:criterionPercolation}]
  We follow closely the proof of Lemma 1 in \cite{Grimmett_Marstrand_1990}. Let $(Y_{\mathsf v})_{\mathsf v \in \mathsf B_n}$ be i.i.d. random variables with the uniform distribution on $[0,1]$. We construct an exploration sequence $\mathsf X'=(\mathsf U',\mathsf V')$ with the same law as the one in the statement of the lemma, in such a way that all vertices $\mathsf v$ ultimately in $\mathsf V'=\mathsf V_\infty'$ satisfy $Y_{\mathsf v}>\zeta$ (these vertices are called ``red'' vertices). Define \[\mathsf X_0'=(\mathsf U_0',\mathsf V_0'):=(\mathsf S,\emptyset),\] and declare all the vertices of $\mathsf S$ to be ``green''. Let $\mathsf z$ be the minimum vertex that belongs to $\partial^+ \mathsf U_0$ (notice that $\mathsf z$ is always well-defined at this first step if $n\ge1$). Then define
$\zeta_{\mathsf z}=\mathbb P(X_{1}=(\mathsf S\cup\{\mathsf z\},\emptyset)\: |\: \mathsf X_0=(\mathsf S,\emptyset))$ and set

  \begin{equation}
        \mathsf X_1'=
    \begin{cases}
      (\mathsf S\cup \{z\}, \emptyset)& \text{if }Y_{\mathsf z}\le\zeta_{\mathsf z}\\
        (\mathsf S, \{z\})& \text{if }Y_{\mathsf z}> \zeta_{\mathsf z}.
      \end{cases}
    \end{equation}
   Declare the vertex $\mathsf z$ to be ``green'' in the first case, and ``red'' in the second case.
   The conditional probability $\zeta_{\mathsf z}$ above is defined in such a way that $(\mathsf X_0', \mathsf X_1')$ has exactly the same law as the  first two steps of the random exploration sequence  of the lemma.

    We proceed iteratively. Assume that $\mathsf X_0'=(\mathsf U_0',\mathsf V_0'),\ldots, \mathsf X_i'=(\mathsf U_i',\mathsf V_i')$ have been constructed. Let $\mathsf z$ be the minimum vertex that belongs to $\partial^+ \mathsf U_i'$ (if $\partial^+ \mathsf U_i'$ is empty then we set $\mathsf X_{i+1}'=(\mathsf U_{i+1}',\mathsf V_{i+1}')=(\mathsf U_{i}',\mathsf V_{i}')$). In case $\mathsf z$ is well defined, consider
    $\zeta_{\mathsf z}=\mathbb P(\mathsf X_{i+1}=(\mathsf U_i\cup\{\mathsf z\},\mathsf V_i)\: |\: \mathsf X_0=\mathsf X'_0,\ldots,\mathsf X_i=\mathsf X'_i)$ and set
  \begin{equation}
    \label{eq:4}
    \mathsf X_{i+1}'=
    \begin{cases}
      (\mathsf U_i'\cup \{z\}, \mathsf V_i')& \text{if }Y_{\mathsf z}\le\zeta_{\mathsf z}\\
        (\mathsf U_i',\mathsf V_i'\cup \{z\})& \text{if }Y_{\mathsf z}> \zeta_{\mathsf z}.
      \end{cases}
    \end{equation}
    As in the first step, declare the vertex $\mathsf z$ to be ``green'' in the first case and ``red'' in the second case.

    Once the algorithm above terminates (when $\mathsf X_{i+1}' = \mathsf X_{i}'$), declare all the vertices that have not been explored to be ``green''.

    Since at each step, the uniform random variable $Y_{\mathsf z}$ is independent of the previous steps, the process $\mathsf X'$ has the same law as $\mathsf X$. Furthermore, since for every explored vertex $\mathsf z$ we have  $\zeta_{\mathsf z}\ge\zeta$ (by hypothesis), we see that the set of green vertices dominates a Bernoulli site percolation with parameter $\zeta$. By Proposition~\ref{prop:oriented}, if $\zeta$ is close enough to $1$, there exist more than $\frac1{1000}|\mathsf R_n|$ vertices of $|\mathsf R_n|$ that can be reached from $\mathsf S$ by a green oriented path with probability larger than $1-e^{-chn}$. By definition of the exploration sequence, all these vertices have been reached by the exploration sequence. This concludes that
    \begin{equation}
      \label{eq:11}
      |\mathsf U_\infty'\cap \mathsf R_n|\ge\frac 1{1000}|\mathsf R_n|
    \end{equation}
    with probability larger than $1-e^{-chn}$.
\end{proof}

\subsection{Proof of Proposition \ref{prop:En}}
Fix $p \in (p_c^{\textrm{site}}(\ZZ^d), 1 - p_c^{\textrm{site}}(\ZZ^d))$. Let $\delta=\delta(d,p)$ be as in Proposition~\ref{prop:k_blocks}. Let $\zeta<1$ as in Proposition~\ref{prop:oriented}. Finally let $k$ be as in  Proposition~\ref{prop:k_blocks}.
We fix $n \geq 3$ and \textbf{assume for convenience that $n=4n'+3$ for some $n'\in \mathbb N$} (we leave the reader convince themself that this hypothesis can be removed). Define the following subsets of the microscopic graph ${\mathbb S}_h$:
\begin{equation}
  \begin{split}
    B_n & := \mathbb Z^d \cap \big( [kn, 2kn] \times (-2kn, 2kn] \times (0, hk] \times (-k, k]^{d-3} \big),\\[2mm]
    L_n & := \mathbb Z^d \cap \big( \{kn\} \times (-kn, kn] \times (0, hk] \times (-k, k]^{d-3} \big),\\[2mm]
   \text{and } R_n & := \mathbb Z^d \cap \big( \{2kn\} \times (-kn-3k, kn+3k] \times (0, hk] \times (-k, k]^{d-3} \big).
  \end{split}
\end{equation}
A picture of the above sets would look very similar to that of the macroscopic $\mathsf{B}_n$, $\mathsf{L}_n$ and $\mathsf{R}_n$ in Figure~\ref{f:B_L_R_n}. Notice that
\begin{equation}
 \bigcup_{\mathsf u\in \mathsf B_n}B^{\mathsf u} \subset B_n,\quad  L_n\subset \bigcup_{\mathsf u\in \mathsf L_n}F^{\mathsf u} \quad\text{ and }\quad\bigcup_{\mathsf u\in \mathsf \partial^+\mathsf R_n}F^{\mathsf u}\subset R_n.\label{eq:31}
\end{equation}
Let $\xi \in \Xi$. We first claim that for every fixed set $T \subseteq L_n$ with $|T|\ge 2 \delta|L_n|$, for all $t:T\to [0,\Ct n]$, we have
\begin{equation}
  \label{eq:7}
  \mathbb P_p \bigg(
  \begin{array}{c}
  \exists T' \subseteq R_n \text{ with } |T'| \ge 8 \delta |\partial \Lambda_{2n}|, \: \exists t' : T'\to [0,2\Ct n] \text{ s.t.}\\
    \text{for all } y \in T', \: \exists x \in T \text{ with } x \lra{\xi_{[t_x,t'_y]}} y \text{ in } B_n
  \end{array}
  \bigg)
  \ge 1-e^{-ch n}.
\end{equation}

Before proving \eqref{eq:7}, let us explain how it  yields Proposition~\ref{prop:En}. Assume that $E_m^{n}(\xi)$ occurs and consider a set $T$ as in \eqref{eq:5}. Partition $T$ into four disjoint subsets corresponding to the left, right, top and bottom part of the boundary of $\Lambda_n$. By symmetry, we may assume for example that the right part, denoted by $T_r$, has size larger than $|T|/4\ge 2\delta|L_n|$. Finally, using independence, we can deduce from \eqref{eq:7}, that with probability larger than $1-e^{-chn}$, the set $T_r$ gives rise to a set $T'$ in $\partial \Lambda_{2n}$ of size larger than $8\delta|\partial \Lambda_{2n}|$. This concludes
\begin{equation}
  \label{eq:8}
  \mathbb P_p \big( E_m^{2n}(\xi)\:|\:E_m^n(\xi) \big) \ge 1-e^{-c h n},
\end{equation}
which completes the proof of Proposition~\ref{prop:En}.  The remainder of this section is devoted to proving~\eqref{eq:7}.

We first set some notation. Let $T \subseteq L_n$ with $|T|\ge 2 \delta|L_n|$, and $t:T \to [0,\Ct n]$. Introduce
\begin{equation}
  \label{eq:9}
  \mathsf{T} := \{\mathsf{u} \in \mathsf{L}_n\: :\:  |T\cap F^{\mathsf{u}}|\ge \delta |F|\}.
\end{equation}
Using that  $T$ has a density at least $2 \delta$ in $L_n$, the second inclusion in \eqref{eq:31} implies that  $\mathsf{T}$ has a density at least $\delta$ in $\mathsf{L}_n$, i.e.
\begin{equation}
  \label{eq:10}
  |\mathsf{T}|\ge \delta |\mathsf{L}_n|.
\end{equation}

Let $\omega$ be a configuration of site percolation with parameter $p$ in $B_n$. We want to define a random exploration process in the macroscopic box $\mathsf{B}_{n}$ that will correspond to a $\xi$-exploration in the microscopic box  $B_n$. For $\mathsf{U}\subseteq \mathsf{B}_{n}$, set $B^{\mathsf{U}} := \bigcup_{\mathsf{u}\in \mathsf{U}}B^{\mathsf{u}}$. For $\mathsf{v}=(\mathsf{v}_1,\mathsf{v}_2,\mathsf{v}_3)\in \mathsf{U}^+$, define
\begin{equation}
  \label{eq:16}
  S^{\mathsf{v}}(\mathsf{U}) := \bigcup_{t\le\Ct \mathsf{v}_1 }\ \bigcup_{x\in T}\big\{y\in F^{\mathsf{v}}\::\:x\lra{\xi_{[t_x,t]}}y\text{ in } B^{\mathsf{U}} \big\},
\end{equation}
and for every $y\in S^{\mathsf{v}}(\mathsf{U})$,
\begin{equation}
  \label{eq:17}
  t^{\mathsf{v}}_y(\mathsf{U}) := \min \big\{ t \le  \Ct (n + \mathsf{v}_1) \: :\: \exists x\in T \text{ with } x \lra{\xi_{[t_x,t]}} y \text{ in } B^{\mathsf{U}} \big\}.
\end{equation}

We now construct a random exploration sequence $\mathsf{X}_0=(\mathsf{U}_0,\mathsf{V}_0)$, $\mathsf{X}_1=(\mathsf{U}_1,\mathsf{V}_1) \ldots$ in $\mathsf{B}_n$ that is measurable with respect to the microscopic percolation configuration $\omega$, and such that for each $i \ge 0$,
\begin{equation}
  \label{eq:19}
  \text{for all $\mathsf{v} \in \mathsf{U}_i^+$}, \quad (S^{\mathsf v}({\mathsf U}_i),t^{\mathsf v}(\mathsf U_i))\text{ is a $\delta$-seed for ${\mathsf v}$.}
\end{equation}
 First, define
\begin{equation}
  \label{eq:18}
  {\mathsf U}_0 := \big\{ {\mathsf u}\in{\mathsf T} \: : \: G_{\xi}^{\mathsf u}(T\cap F^{\mathsf u},t|_{T\cap F^{\mathsf u}}) \text{ occurs} \big\} \quad \text{ and } \quad {\mathsf V}_0 := \emptyset
\end{equation}
(here we use that $(T\cap F^{\mathsf u},t|_{T\cap F^{\mathsf u}})$ is a $\delta$-seed for ${\mathsf u}$, from the definition of ${\mathsf T}$ in \eqref{eq:9}). It follows from the definition~\eqref{eq:6} that \eqref{eq:19} is satisfied for $i = 0$.

Now, fix $i\ge0$ and assume that $\mathsf X_0=(\mathsf U_0,\mathsf V_0),\ldots,\mathsf X_i=(\mathsf U_i,\mathsf V_i)$ have been constructed in such a way that \eqref{eq:19} is satisfied for some $i$.
We define $X_{i+1}$ as follows. If $\partial^+\mathsf U_i\cap {\mathsf B}_n\subseteq {\mathsf V_i}$, then set $X_{i+1}=X_i$. Otherwise let ${\mathsf v}$ be the minimum vertex that belongs to $\partial^+\mathsf U_i\cap {\mathsf B}_n$ but not to ${\mathsf V_i}$. By \eqref{eq:19}, the pair $(S^{\mathsf v},t^{\mathsf v})=(S^{\mathsf v}({\mathsf U}_i),t^{\mathsf v}({\mathsf U}_i))$ is a seed for ${\mathsf v}$, and we define
\begin{equation}
  \label{eq:20}
  (\mathsf U_{i+1},\mathsf V_{i+1})=
  \begin{cases}
     \: (\mathsf U_{i}\cup\{\mathsf v\},\mathsf V_{i}) & \text{if $G_{\xi}^{\mathsf v}(S^{\mathsf v},t^{\mathsf v})$ occurs}\\[2mm]
     \: (\mathsf U_{i},\mathsf V_{i}\cup\{\mathsf v\}) & \text{otherwise.}
  \end{cases}
\end{equation}
First, by Hoeffding's inequality, we have
\begin{equation}
  \label{eq:21}
  {\mathbb P}_p(|{\mathsf U}_0|\ge \tfrac\delta 2|{\mathsf L}_n|)\ge 1-e^{-c h n}
\end{equation}
for some constant $c>0$. Using independence and Proposition~\ref{prop:k_blocks}, we also obtain
\begin{equation}
  \label{eq:22}
  {\mathbb P}_p\big (\mathsf V_{i+1}= \mathsf V_i  \:|\:\mathsf X_0,\ldots,\mathsf X_i\big)\ge \zeta \text{ a.s.}
\end{equation}
The two equations above, combined with Lemma~\ref{lem:criterionPercolation}, imply that
\begin{equation}
  \label{eq:23}
  \mathbb P_p\big(|{\mathsf U}_\infty\cap {\mathsf R}_n|\ge \tfrac{1}{1000} |{\mathsf R}_n|\big)\ge 1-2 e^{-c h n}.
\end{equation}
When the exploration terminates, the property \eqref{eq:19} is still satisfied. Therefore, on the event $|{\mathsf U}_\infty\cap {\mathsf R}_n|\ge \tfrac{1}{1000} |{\mathsf R}_n|$, the set
\begin{equation}
  \label{eq:24}
  T'=\bigcup_{{\mathsf v} \in {\mathsf U}_\infty^+ \cap \partial^+{\mathsf R}_n}S^{\mathsf v}({\mathsf U}_\infty)
\end{equation}
satisfies
\begin{itemize}
\item $|T'|\ge 64000 \delta|F| \cdot |{\mathsf U}_\infty^+ \cap \partial^+ {\mathsf R}_n| \ge 64 \delta |F|\cdot|{\mathsf R}_n| = 64 \delta|R_n|\ge 8\delta|\partial \Lambda_{2n}|$,\\[-5mm]
\item and for each $y$ in $T'$, there exist $t'_y\le 2\Ct n$ and $x \in T$ such that $x\lra{\xi_{[t_x,t'_y]}}y$ in $B_{n}$.
\end{itemize}
We can thus deduce \eqref{eq:7} from \eqref{eq:23}, which completes the proof.

\section{Proof of Wierman's coupling} \label{sec:Wierman}

\begin{figure}
\begin{center}

\subfigure{\includegraphics[height=.4\textwidth]{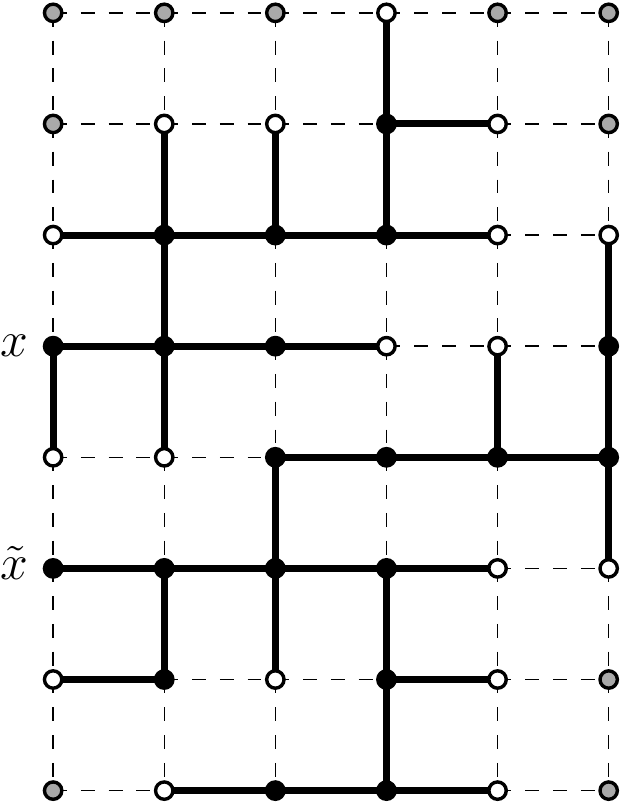}}
\hspace{0.16\textwidth}
\subfigure{\includegraphics[height=.4\textwidth]{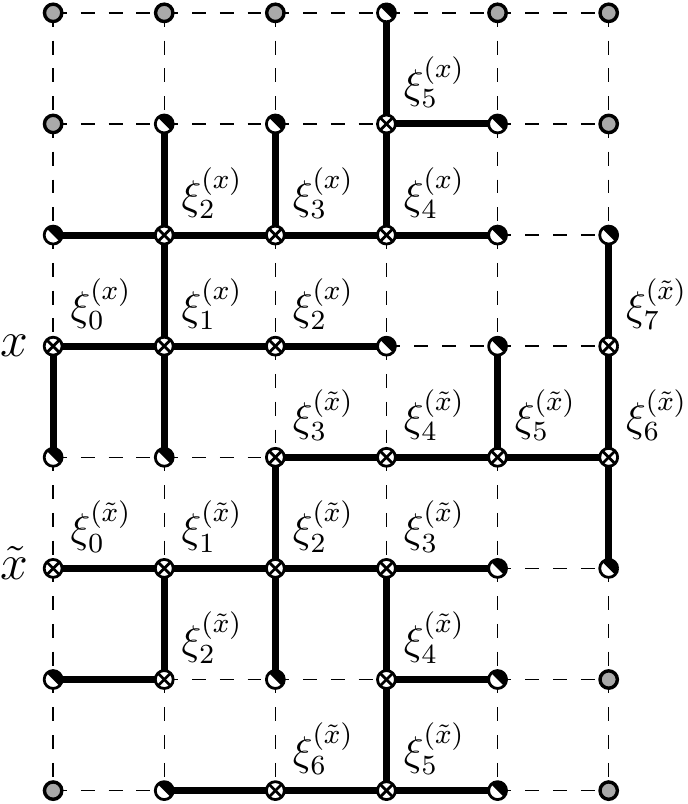}}\\

\caption{\label{fig:Wierman} This figure shows the coupling between two configurations $\omega$ (\emph{left}) and $\tilde{\omega}$ (\emph{right}) used in the proof of Lemma \ref{lem:Wierman_gen}. Each of these configurations is distributed as Bernoulli site percolation with parameter $p \leq \frac{1}{2}$. Starting an exploration procedure from the vertices in $S = \{ x, \tilde{x} \}$, we produce a spanning forest of the $1$-cluster of $S$, together with its outer boundary: its branches are used to read $\wordone$ on $\omega$, and $(\xi^{(x)}, \xi^{(\tilde{x})})$ on $\tilde{\omega}$. The $0$-sites in $\omega$ discovered during the exploration procedure may be $0$ or $1$ in $\tilde{\omega}$, and they are not used to read $\xi^{(x)}$ or $\xi^{(\tilde{x})}$. The gray vertices remain undiscovered.}
\end{center}
\end{figure}

\begin{proof}[Proof of Lemma \ref{lem:Wierman_gen}]
In the configuration $\omega$, we are going to explore the set of vertices which are $1$-connected to $S$ in $\Lambda$, as well as the immediate neighborhood of this set. During this exploration, we grow (in a Markovian way) a spanning forest of this explored set, such that at all times, it contains exactly one tree for each vertex $x \in S$, as illustrated in Figure~\ref{fig:Wierman}.

  We start by fixing an arbitrary order on the vertices of~$\Lambda$.
  We then let the first forest $\mathcal{F}_0$ in our construction consist simply of the vertices in $S$, without any edges.
  Having defined $\mathcal{F}_t$ for some time $t$, we inductively add one extra vertex and one extra edge to $\mathcal{F}_t$ in order to obtain $\mathcal{F}_{t + 1}$.
  For this, we pick the minimum unexplored vertex $y'$ that neighbors a $1$-vertex of $\mathcal{F}_t$, and we let $y$ be the minimum $1$-vertex in $\mathcal{F}_t$ neighboring $y'$.

We now add the vertex $y'$ and the edge $(y, y')$ to $\mathcal{F}_{t + 1}$ as follows: let $d$ be the distance between $y$ and $S$ within the forest, and set
\begin{equation}
(\omega_{y'}, \tilde{\omega}_{y'}) = \left \lbrace
\begin{array}{ll}
(1, \xi^{(x)}_{d+1}) & \text{w.p. } p,\\[2mm]
(0,0) & \text{w.p. } 1 - 2p \text{ if } \xi^{(x)}_{d+1} = 0, \text{ and w.p. } 1 - p \text{ if } \xi^{(x)}_{d+1} = 1,\\[2mm]
(0,1) & \text{w.p. } p \text{ if } \xi^{(x)}_{d+1} = 0
\end{array}
\right.
\end{equation}
(the third option does not arise when $\xi^{(x)}_{d+1} = 1$). When the exploration procedure stops, we simply draw independently the remaining undiscovered vertices.

By construction, all the sites $y$ which are $1$-connected to $S$ are $\xi^{(x)}$-connected to some $x \in S$ in $\tilde{\omega}$, by following the same branch of the spanning forest, which ensures that the property in the statement holds true. Note that the sites which are discovered and have state $0$ in $\omega$ are not used here to observe $(\xi^{(x)})_{x \in S} \in \Xi$.
\end{proof}

\section{Oriented percolation lemma} \label{sec:oriented}

In this section, we establish Proposition~\ref{prop:oriented}, which is a result about oriented percolation on the slab $\overrightarrow{\mathsf{Slab}}_h$. Before proving it, we establish an auxiliary lemma for planar oriented percolation, i.e.\@ in dimension $d = 1 + 1$ (for more details on the topic, see e.g.~\cite{MR757768}). We consider the oriented graph~$G$ (illustrated on Figure~\ref{fig:1}) with vertex set
\begin{equation}
  V=\big\{(x,y)\in \mathbb Z^2\: :\: x\in 2\mathbb Z,\, \tfrac x2 +y\in 2\mathbb Z\big\},  
\end{equation}
and (oriented)  edge set 
\begin{equation}
   E=\big\{(u,u+(2, e)) \: :\: u\in V,\, e=\pm 1\big\}.
 \end{equation}

 \begin{figure}[htbp]\label{fig:1}
  \centering
  \includegraphics[width=3cm]{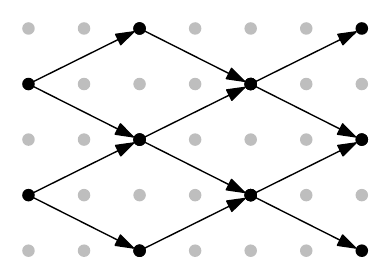}
  \caption{A portion of the  graph $G=(V,E)$.}
\end{figure}
This graph, which is isomorphic to the usual oriented square lattice, can also be seen as the projection of $ \overrightarrow{\mathsf{Slab}}_h$ onto its first two coordinates. This will be convenient later because we will see  $G$ as a subgraph of the oriented slab $\overrightarrow{\mathsf{Slab}}_h$.

Consider the Bernoulli site  percolation measure on $G$ with parameter $0\le \gamma \le 1$. Each vertex in $V$ is open with probability $\gamma$ and closed with probability $1-\gamma$, independently of the other vertices.  Let $n$ be an even integer. For a set of vertices $\mathsf A\subseteq \{0\}\times\mathbb Z$, define
 \begin{equation}
   \xi^\mathsf{A}_{5n} := \big\{ y \in [-n,n] \: : \: \mathsf{A} \to (5n, y) \big\},
 \end{equation}
where $ \mathsf{A} \to v$ means that there exists an open oriented path from a vertex $u\in \mathsf A$ to $v$. Notice that the quantity above satisfies the following monotonicity property:
\begin{equation}
  \label{eq:28}
  \text{If $\mathsf A \subseteq \mathsf B$, then }\xi^\mathsf{A}_{5n} \subseteq \xi^\mathsf{B}_{5n}.
\end{equation}

\begin{lemma}
  \label{l:domination}
  $(d = 1 + 1)$ For any $0<\delta<\frac{1}{10}$, there exists $\gamma = \gamma(\delta) \in (0, 1)$ large enough so that: for every $n\ge2$ even, for every set of vertices $\mathsf{S} \subseteq \{0\} \times [-n, n]$ with $|\mathsf{S}| \geq \delta n$, we have that $\xi^\mathsf{S}_{5n}$ stochastically dominates the product measure on $\{-n,-n+2, \ldots, n\}$ with parameter $\frac{1}{2}$.
\end{lemma}

\begin{proof} Let $\varepsilon>0$.  By Theorem~1.1 from \cite{Liggett_Steif_2006}, we can choose $\gamma$ close enough to 1 such that
   \begin{equation}
    \label{e:stationary_dominates}
    \xi_{5n}^{\mathbb Z} \text{ stochastically dominates the product measure with parameter } 1-\varepsilon.
  \end{equation}
    where we use the notation  $\xi_{5n}^{\mathbb Z}=\xi_{5n}^{\{0\}\times\mathbb Z}$.
  The main idea here is to use the stochastic domination above, and replace $\mathbb Z$ by $\mathsf S$. To achieve this, we construct an increasing  event $G$ such that
  \begin{enumerate}[(i)]
  \item \label{item:1} $G$ occurs with probability exponentially close to $1$,
  \item \label{item:2} on the event $G$, we have $\xi_{5n}^{\mathsf S}= \xi_{5n}^{\mathbb Z}$.
  \end{enumerate}
  These  properties together with \eqref{e:stationary_dominates} easily conclude the proof (the argument is presented at the end of the proof).   Let us now  construct the event $G$.  Set $\mathsf{S}' := \mathsf{S} \cap \big( \{0\} \times \big[ -n + \frac{\delta}{4}  n, n - \frac{\delta}{4} n \big] \big)$ and note that since $|\mathsf{S}| \geq \delta n$, we have $|\mathsf{S}'| \geq \frac{\delta}{2} n$.  Let $G$ be the event (illustrated on Figure~\ref{fig:2}) that
  \begin{itemize}
  \item $\mathsf S'\to v$ for some $v\in \{5n\}\times \mathbb Z$,
  \item there is an open oriented path from $\{0\}\times[-n, -n + \frac{\delta}{4} n]$ to  $\{5n\}\times [n,\infty)$,
  \item there is an open oriented path from $\{0\}\times[n - \frac{\delta}{4} n, n]$ to  $\{5n\}\times (-\infty, -n]$.
  \end{itemize}
  The first item occurs with probability larger than $1-\varepsilon^{\delta n/2}$. To see this, consider  the set $X$ of vertices $u\in \{0\}\times\mathbb Z$ that are connected to some $v\in  \{5n\}\times\mathbb Z$. Then use the  stochastic domination \eqref{e:stationary_dominates} and symmetry to show that $X$ stochastically dominates a product measure with parameter $1-\varepsilon$. The second and third item can be proved to occur with probability larger than  $1-\tfrac13 4^{-n}$ by choosing $\gamma$ close enough to 1. This follows from choosing the density of open vertices close enough to 1 in the proof of Theorem~4 in \cite{Durrett_Griffeath_1983}. Alternatively, this can also be proved directly using a Peierls-type argument. Therefore, if $\varepsilon>0$ is chosen small enough, we obtain that
  \begin{equation}
    \label{eq:27}
    \mathbb P(G) \ge 1-4^{- n},
  \end{equation}
  which establishes Item~(\ref{item:1}) above.
  
\begin{figure}[htbp]\label{fig:2}
  \centering
  \includegraphics[width=7.6cm]{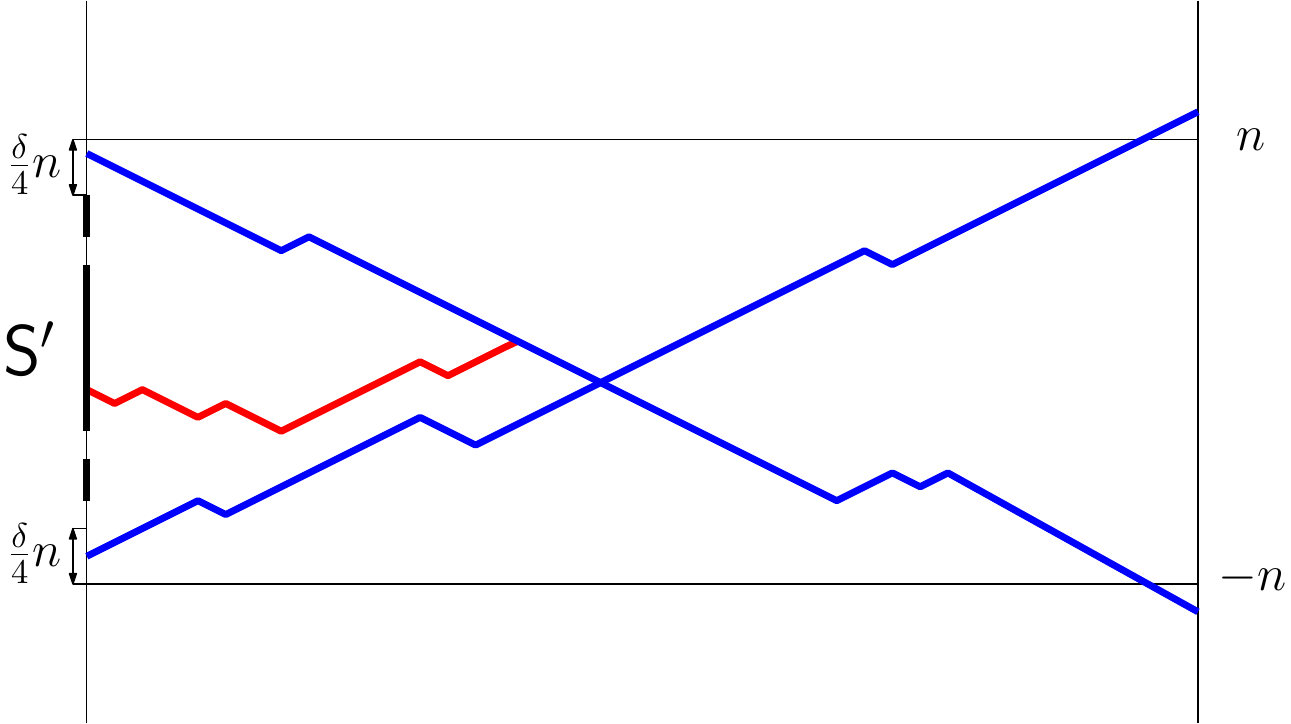}
  \caption{Schematic representation of the event $G$}
\end{figure}

Now, assume that the event $G$ occurs, and consider a vertex $v$ on $\{5n\}\times [-n,n]$ that can be reached by an open path starting on $\{0\}\times \mathbb Z$.
Then, from the three oriented paths that appear in the definition of $G$, one can produce an open oriented path from $\mathsf S'$ to $v$. In other words, when $G$ occurs, we have $\xi_{5n}^{\mathbb Z} \subseteq \xi_{5n}^{\mathsf S'}$. Since $\mathsf S' \subseteq \mathsf S$, the monotonicity property~\eqref{eq:28} directly implies that
\begin{equation}
  \label{eq:26}
  \xi_{5n}^{\mathbb Z} \subseteq \xi_{5n}^{\mathsf S}.
\end{equation}
By monotonicity again, the reverse inclusion also holds, which establishes Item~(\ref{item:2}). Once we know the stochastic domination \eqref{e:stationary_dominates} and the existence of an increasing  event $G$ satisfying (\ref{item:1}) and (\ref{item:2}) above, the proof of the lemma follows from a general stochastic domination argument that we detail now.

We consider an auxiliary product random variable $\eta=(\eta_{-n},\eta_{-n+2},\ldots,\eta_n)$ independent of the percolation configuration, where the $\eta_i$'s are i.i.d.~Bernoulli random variables with parameter~$\frac{3}{4}$. In order to prove the lemma, it is enough to show that $\xi^{\mathsf{S}}_{5n}$ stochastically dominates $\eta\cdot\xi^{\mathbb Z}_{5n}$ (where $\eta_1 \cdot \eta_2$ denotes the coordinate-wise product of $\eta_1$ and $\eta_2$). Let $f:\{0, 1\}^{\{-n,-n+2,\ldots,n\}}\to \mathbb R_+$  be a non-decreasing function such that $f(\mathbf 0)=0$.
We have
\begin{align*}
  \mathbb E[f(\xi^{\mathsf{S}}_{5n})]&\overset{\ \: f\ge0 \ \:}\ge  \mathbb E[f(\xi^{\mathsf{S}}_{5n})\mathbf 1_G] \\
                                               & \overset{\ \ (\ref{item:2}) \ \ }{=} \mathbb E[f(\xi^{\mathbb Z}_{5n})\mathbf 1_G]\\ &\overset{ \ \, \text{FKG} \ \, }{\ge} \mathbb E[f(\xi^{\mathbb Z}_{5n})]  \mathbb P (G)\\
                                     & \overset{\ \: \eqref{eq:27} \: \ }\ge \mathbb E[f(\xi^{\mathbb Z}_{5n})]\mathbb P(\eta\neq \mathbf 0)\\ & \overset{\ \text{indep.}}=  \mathbb E[f(\xi^{\mathbb Z}_{5n})\mathbf 1_{\eta\neq \mathbf 0}]\\& \overset{f(\mathbf 0)=0}\ge \mathbb E[f (\eta\cdot\xi_{5n}^{\mathbb Z})].                                  
\end{align*}
This proves the desired stochastic domination and concludes the proof. 
\end{proof}

Before moving to the proof of  Proposition~\ref{prop:oriented}, we use the planar result above to deduce a useful statement for oriented percolation on the  slab. Fix $h \geq 2$, and consider Bernoulli site  percolation  with parameter $0\le \gamma<1$ on the slab $\overrightarrow{\mathsf{Slab}}_h = (\mathsf V, \vec{\mathsf E})$ (where $\mathsf V \subseteq \mathbb{Z}^2 \times [0, h]$), as defined in \eqref{eq:vertices_slab} and \eqref{eq:oriented_edges}. For any given $n \ge m \geq h$, we also define the sets
\begin{equation}
  \begin{split}
    \mathsf{B}_{n, m} & := \mathsf{V} \cap \big( (n, 2n) \times (-2m, 2m) \times (0, h) \big),\\[2mm]
    \mathsf{L}_{n, m} & := \mathsf{V} \cap \big( \{n+1\} \times [-m ,m] \times (0, h) \big),\\[2mm]
   \text{and } \mathsf{R}_{n, m} & := \mathsf{V} \cap \big( \{2n-1\} \times [-m, m] \times (0, h) \big).
  \end{split}
\end{equation}
We extend the definition to non integers $m$ by setting $\mathsf{B}_{n, m}:=\mathsf{B}_{n, \lceil  m \rceil}$ (and similarly for $\mathsf{L}_{n, m}$ and $\mathsf R_{n, m}$).
Note that these sets are similar to those defined in \eqref{eq:macro_sets}, except that they are shorter in the $y$-direction.   For $\mathsf A \subset  \mathsf B\subset  \mathsf  V$ and $\mathsf v\in \mathsf B$, we say that  $\mathsf A \xrightarrow{}{} \mathsf v$ \emph{in $\mathsf B$}  if there exists an open  oriented path of $\overrightarrow{\mathsf{Slab}}_h$ that stays in $\mathsf B$ from a vertex of $\mathsf A$ to $\mathsf v$.

\begin{lemma} \label{l:dominationbis}
    For every $\delta > 0$, there exists $\gamma=\gamma(\delta) \in (0, 1)$ (large enough)  such that  for every $n \geq h\ge 6$ and $\mathsf{S} \subseteq \mathsf{L}_{n,  n/h}$ with $|\mathsf{S}| \geq \delta n$, we have 
  \begin{equation}
    \mathbb P \bigg( N > \frac 1 {20}n \bigg) \ge 1- e^{-cn}
  \end{equation}
  for some $c >0$, where $N$ is the number of vertices $\mathsf{u} \in \mathsf{R}_{n,  n/h }$  such that $\mathsf{S} \xrightarrow{}{} \mathsf{u}$  in $\mathsf{B}_{n, n/h }$.
\end{lemma}

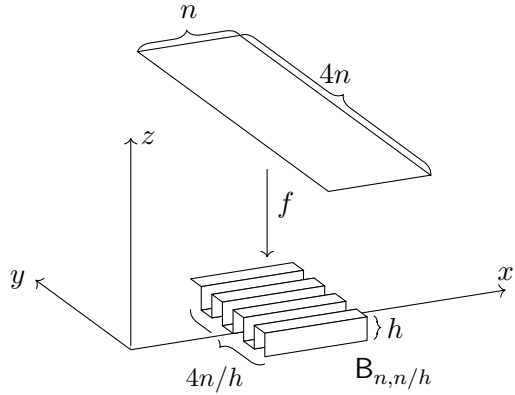
\begin{figure}[ht]
  \centering
  \tdplotsetmaincoords{70}{110}
  \begin{tikzpicture}[scale=.5,tdplot_main_coords]
    \tdplotsetrotatedcoords{0}{0}{135}
    \begin{scope}[tdplot_rotated_coords]
      \draw[->] (-3,0,0) -- (8,0,0) node[anchor=south]{$x$};
      \draw (0,-6,6) -- (0, 6, 6) -- (3, 6, 6) -- (3, -6, 6) -- cycle;
      \draw[decorate,decoration={brace,amplitude=5pt}] (0, 6, 6) -- (3, 6, 6) node [black, midway, yshift=12pt] {$n$};
      \draw[decorate,decoration={brace,mirror,amplitude=5pt}] (3, -6, 6) -- (3, 6, 6) node [black, midway, yshift=12pt] {$4n$};
      \foreach \n in {9, ..., 3}
      {
        \fill[rounded corners=.02, white,draw=black]
        (0, 2*\n/3 - 4, 2/3 * \modulo{\n}{2}) -- (3, 2*\n/3 - 4, 2/3 * \modulo{\n}{2})
        -- (3, 2*\n/3 - 10/3, 2/3 * \modulo{\n}{2}) -- (0, 2*\n/3 - 10/3, 2/3 * \modulo{\n}{2}) -- cycle;
        \fill[rounded corners=.02, white,draw=black]
        (0, 2*\n/3 - 4, 0) -- (3, 2*\n/3 - 4, 0) -- (3, 2*\n/3 - 4, 2/3) -- (0, 2* \n/3 - 4, 2/3) -- cycle;
      }
      \draw[decorate,decoration={brace,mirror,amplitude=2.5pt}] (3.2, -2, 0) -- (3.2, -2, 2/3) node [black, midway, xshift=8pt] {$h$};
      \draw[decorate,decoration={brace,amplitude=5pt}] (0, -2, -.2) -- (0, 2 + 2/3, -.2) node [black, midway, yshift=-16pt,xshift=-5pt] {\small $4  n/h $};
      \draw[->] (-3,0,0) -- (-3,6,0)  node[anchor=east]{$y$};
      \draw[->] (-3,0,0) -- (-3,0,6)  node[anchor=west]{$z$};
      \draw[->] (1, 0, 4.5) -- (1, 0, 2);
      \node[anchor=west] at (1, 0, 3.5) {$f$};
      \draw[line width=3pt,color=white] (-3,0,0) -- (-.4,0,0);
      \draw (-3,0,0) -- (-.4,0,0);
      \node[yshift=-8pt,xshift=4pt] at (3, -3, 0) {$\mathsf{B}_{n,  n/h }$};
    \end{scope}
  \end{tikzpicture}
  \caption{The embedding used in the proof of Lemma~\ref{l:dominationbis}.}
  \label{f:accordion}
\end{figure}

\begin{proof}[Proof of Lemma~\ref{l:dominationbis}]
  For simplicity, we assume that $h$ divides $n$. It is enough to provide a proper graph embedding, or more precisely to find an injective  mapping $f: V\cap( [0, n] \times [-2n, 2n]) \to \mathsf{B}_{n, n/h}$ satisfying:
  \begin{enumerate}[\quad a)]
  \item there are at least $\tfrac \delta {10} n$  vertices  $u \in \{0\} \times [-n, n]$ such that   $f(u) \in \mathsf S$,
  \item for every $u \in \{n\} \times [-n, n]$, we have $f(u) \in \mathsf{R}_{n,  n/h }$,
  \item and if $(u,v)$ is an edge on the oriented graph $(V,E)$, then $(f(u), f(v))$ is an edge in $(\mathsf V, \vec{\mathsf{E}})$.
  \end{enumerate}
  Once we have such an embedding at our disposal, we can simply use Lemma~\ref{l:domination} (combined with Hoeffding's inequality).

  The construction of the embedding $f$ follows from an ``accordion-like'' folding of the rectangle, and then finding a translate of this accordion such that property $a$ above occurs.
  We refrain from giving a detailed definition of the function $f$, and we give instead an illustration of it in Figure~\ref{f:accordion}.
\end{proof}

We are now in a position to prove the main proposition of this section.

\begin{proof}[Proof of Proposition~\ref{prop:oriented}]
  We start the proof by making some assumptions that do not reduce the generality of the result.
  First suppose that the side length of the rectangle is a multiple of $40h$  (this way, all the fractions of $n$ appearing in the proof will be integers).

  Observe that the starting set $\mathsf L_{n}$ can be written as the union
  \begin{equation}
    \label{eq:25}
    \mathsf L_n= I_{-5}\cup\cdots\cup  I_4,
  \end{equation}
  where $I_i:= \mathsf V \cap (\{n\} \times [in/5,in/5+n/5] \times [0, h])$. Since $\mathsf S$ has at least $\delta nh$ points, we deduce that the intersection of $\mathsf S$ with at least one of the $ I_i$'s contains at least $\delta nh/10$ points. Without loss of  generality, we assume  that
  \begin{equation}
    \label{eq:29}
    |\mathsf S \cap I_1| \geq \frac{\delta n h}{10}.
  \end{equation}
  We are going to construct the connections described in Proposition~\ref{prop:oriented} in two steps. Intuitively speaking, we show that the infection spreads from left to right in two steps, first until the middle of the box ($x = 3n/2$), where the infection becomes evenly distributed across the $y$-direction.
  Then from the middle to the end ($x = 2n$), where we show that it spreads across the $z$-direction.
  Let us state precisely these two steps of the proof.
  We first choose a collection $J_1, \ldots, J_{h/30} $ of rectangles of the form
  $$J_i=\mathsf V \cap \bigg( \{3n/2\} \times \bigg[ j_i - \frac n h , j_i +  \frac n h \bigg] \times [0, h] \bigg)$$
  such that the mutual distance between different $J_i$'s is strictly larger than $\frac{2n}{h}$. We say that some set $J_i$ is bad if
  \begin{equation}
    \big| \{u \in J_i \: : \: \mathsf{S} \to u \} \big| \leq \tfrac \delta{100}|J_i|.
  \end{equation}
  In order to establish Proposition~\ref{prop:oriented}, we claim that for some  constant $c > 0$ (depending on~$\delta$)
  \begin{equation}
    \label{e:step_1}
    \mathsf P_\gamma \Big( \text{at least $\tfrac h {60}$ sets $J_i$ are bad} \Big) \leq e^{-c h n}
  \end{equation}
(this claim will be established at the end of the proof).  In the final step of the proof, we assume that at least $\frac{h}{60}$ sets $J_i$ are good. For each good $J_i$, we consider the set $S_i$ (with density at least $\frac{\delta}{100}$ in $J_i$) of vertices in $J_i$ which can be reached from $\mathsf S$ by an open oriented path, and we work in the box $K_i=[\tfrac32n,n]\times [ j_i - \tfrac {2n} h , j_i +  \frac {2n} h \big] \times [0, h]$. Using Lemma~\ref{l:dominationbis}, we can  show that with probability larger than $1-e^{-cn}$, we can reach at least $\frac1{20}n$  points in $\mathsf R_n$ with oriented paths starting in $S_i$ and staying in $K_i$. Since the boxes $K_i$ are disjoint, this construction is successful for at least $\frac{h}{100}$ indices $i$, with probability larger than $1-e^{-c'hn}$  (by  Hoeffding's inequality). Therefore,
  \begin{equation}
    \label{eq:30}
    \mathsf{P}_\gamma \Big( \big| \big\{ \mathsf v \in \mathsf{R}_n \: : \: \mathsf{S} \xrightarrow{}{} \mathsf v \big\} \big| \geq
    \frac h {100}\cdot \frac1{20} n \Big) \geq 1- e^{- c h n}-e^{-c'hn},
  \end{equation}
  which  implies Proposition~\ref{prop:oriented}.

  Let us now prove \eqref{e:step_1}. Assume for simplicity that $h$ is odd.  We split the set $I_1$ into its rows $I_1 = \cup_{i = 0,2, \ldots, h-1} C_i$, where $C_i$ is defined as the set of points in $I_1$ with $z=i$.
  We can bound the size of $\mathsf{S}$ intersected with $I_1$ as follows:
  \begin{equation}
    |\mathsf{S}\cap I_1 | \leq n \cdot \Big| \Big\{ i \: : \: |\mathsf{S} \cap C_i| \geq \frac{\delta n}{20} \Big\} \Big| + \frac{\delta n  h}{20}.
  \end{equation}
  Since the left-hand side is larger than $\delta nh/10$ (see Eq.~\eqref{eq:29}), we deduce 
  \begin{equation}
    \label{e:size_S_prime}
    \Big| \Big\{ i\: : \: |\mathsf{S} \cap C_i| \geq \frac{\delta n}{20} \Big\} \Big| \geq \frac{\delta h}{20}.
  \end{equation}
  We denote the above set of good columns by $G \subseteq \{0,2, \ldots, h-1\}$.   For each good column $C_i$, $i \in G$, we restrict our oriented percolation to the sheet that is isomorphic to the oriented square lattice, and is contained in $\mathbb{Z} \times \mathbb{Z} \times \{i, i + 1\}$.
  This restriction can only reduce the set of open paths and it has two advantages: distinct sheets become independent, and it allows us to use the planar result Lemma~\ref{l:domination}.

  Consider the set $\mathsf{S}' = \mathsf V\cap(\{3n/2\} \times [0,n/10] \times G)$.
  Applying Lemma~\ref{l:domination}  produces $\gamma$ for which
    $$\big\{ u \in \mathsf{S}' \: : \: \mathsf{S} \xrightarrow{\mathsf B_n}{} u \big\} \text{ stochastically dominates the product measure on } \mathsf{S}' \text{ w.p. } \frac{1}{2}.$$
  This implies that for each $i$, since $|J_i\cap \mathsf{S}'| \ge \frac \delta{20} |J_i\cap \mathsf V| $, we have
  $$P \big( J_i \text{ is bad} \big) \leq P \Big( \Bin \Big(  \tfrac \delta{20} |J_i\cap \mathsf V| , \tfrac{1}{2} \Big) \leq  \tfrac \delta{100} |J_i\cap \mathsf V|  \Big) \leq e^{-c_0 n}$$
  for some constant $c_0>0$ (using again Hoeffding's inequality). Hence,
      $$P \Big( \text{at least } \tfrac{h}{60} \text{ sets } J_i \text{ are bad} \Big) \leq 2^{h/30} \sup_{\substack{D \subseteq \{0, \ldots, h/30\}\\ |D| \geq \frac{h}{60}}} P \big( J_i \text{ is bad for all } i \in D \big) \leq 2^{h/30} \big( e^{-c_0n} \big)^{h/60},$$
  which  establishes \eqref{e:step_1}.
\end{proof}

\section{Open problems}
\label{sec:open}

We state two open problems that follow naturally from the results presented in this work. The first problem is related to local uniqueness for percolation of words.

\begin{problem}
For percolation of words on $\ZZ^d$, $d \geq 3$, at $p \in (p_c^{\textrm{site}}(\ZZ^d), 1 - p_c^{\textrm{site}}(\ZZ^d))$, consider the following event. For $n \geq 1$, $U_n := \{$there exist $x, y \in B_n(0)$ and $\xi \in \Xi$ for which $x \zgg{\xi} \infty$ and $y \zgg{\xi} \infty$, but there does not exist any way to see $\xi$ starting from $x$ and $y$ that ``coalesces'' before exiting $B_{2n}(0) \}$. More precisely, we require $x$, $y$ and $\xi$ to satisfy: if $\gamma$ and $\gamma'$ are two infinite self-avoiding paths from $x$ and $y$, respectively, along which $\xi$ is seen, then $\gamma$ and $\gamma'$ do not ``meet'' (i.e.\@ intersect each other at the same time) before exiting $B_{2n}(0)$. What is the rate of decay of $\PP_p(U_n)$ as $n \to \infty$?
\end{problem}

In a possible generalization of words to higher-dimensional objects, we can ask the following question.

\begin{problem}
  What can be said about ``percolation of images'', i.e.\@ when is it possible to embed, in a Lipschitz way, all (or almost all) elements of $\{0, 1\}^{\NN^2}$ (or $\{0, 1\}^{\ZZ^2}$) into a configuration of Bernoulli site percolation on $\ZZ^d$, $d \geq 3$?
\end{problem}

Let us also mention that the construction in Section~4 of \cite{Benjamini_Kesten_1995} shows the following. For any fixed $p \in (0,1)$, there exists $d = d(p)$ large enough such that in $\ZZ^d$: $\PP_p$-a.s., all words can be seen from the neighbors of a single vertex. Our proofs show that for all $d \geq 3$ and $p \in (p_c^{\textrm{site}}(\ZZ^d), 1 - p_c^{\textrm{site}}(\ZZ^d))$, there exists $n = n(d,p)$ large enough so that: $\PP_p$-a.s., there exists a ball $B_n(v)$ ($v \in \ZZ^d$) of radius $n$ from which all words can be seen. However, the value of $n$ produced by the proof, using a renormalization procedure, is typically very large, and we cannot say anything about what the optimal $n$ should be. For example, it leaves completely open the question of whether $d(\frac{1}{2}) = 3$, i.e.\@ all words can already be seen from the neighbors of a single vertex on $\ZZ^3$, at $p = \frac{1}{2}$.

For more open questions and comments regarding percolation of words, the reader can consult the classification conjecture of de Lima and Sidoravicius, stated in Section 1.2 of \cite{Hilario_Lima_Nolin_Sidoravicius_2014}. In the present paper, we completely verified part II.(c) of this conjecture in the case of $\ZZ^d$, $d \geq 3$. Also, since Wierman's coupling is fully general, our proof suggests that II.(c) should be valid for a rather wide family of graphs, on which the renormalization procedure can be carried out.

\bibliographystyle{plain}
\bibliography{Percolation_words}

\end{document}